\documentclass[11pt]{article}
\usepackage[utf8]{inputenc}

\usepackage{geometry}
\usepackage{mathtools,amsfonts,amsthm,mathrsfs,amssymb,stmaryrd, bbm}
\SetSymbolFont{stmry}{bold}{U}{stmry}{m}{n}
\usepackage[dvipsnames]{xcolor}
\usepackage{tikz}
\usepackage{standalone}
\usepackage[justification=centering]{caption}
\usepackage{comment}
\usepackage[colorlinks=true, citecolor=cyan, linkcolor=purple]{hyperref}

\newtheorem{thm}{Theorem}[section]
\newtheorem{lemma}[thm]{Lemma}
\newtheorem{corol}[thm]{Corollary}
\newtheorem{prop}[thm]{Proposition}
\newtheorem{conj}{Conjecture}

\newtheorem{defin}[thm]{Definition}
\newtheorem{rem}[thm]{Remark}
\newtheorem{example}[thm]{Example}

\renewcommand{\S}{\mathfrak{S}}
\newcommand{\Z}{\mathbb{Z}}
\newcommand{\N}{\mathbb{N}}
\newcommand{\B}{\mathcal{B}}

\newcommand{\x}{r_{\sigma}}
\newcommand{\y}{\ell_{\tau}}
\newcommand{\Rs}{R_{\sigma}}
\newcommand{\Lt}{L_{\tau}}
\newcommand{\move}{\rightarrow}
\newcommand{\orb}{\mathcal{O}}
\newcommand{\id}{{\text{id}}}
\newcommand{\desc}{\overline{\text{id}}}
\newcommand{\intval}[1]{\llbracket #1 \rrbracket}
\newcommand{\vsum}{\obar}
\newcommand{\hsum}{\ominus}
\newcommand{\dsum}{\obslash}
\newcommand{\usum}{\odot}
\newcommand{\filling}{\varphi}

\definecolor{colorx}{RGB}{136, 189, 30}
\definecolor{colory}{RGB}{212, 43, 11}
\definecolor{colorz}{RGB}{20, 65, 204}
\newcommand{\blue}[1]{\textcolor{colorz}{#1}}
\newcommand{\red}[1]{\textcolor{colory}{#1}}
\newcommand{\green}[1]{\textcolor{colorx}{#1}}

\newcommand{\ddincr}{(\blue{12}, \red{12})}
\newcommand{\pat}{(\blue{312}, \red{231})}
\newcommand{\sigtau}{(\blue{\sigma}, \red{\tau})}
\newcommand{\colpair}[2]{(\blue{#1}, \red{#2})}

\newcommand{\A}{Av_n(\ddincr, \pat)}

\newcommand{\Psibis}{\widehat{\Psi}}
\newcommand{\loz}{\mathcal{L}}
\newcommand{\Yline}{\mathcal{Y}}
\newcommand{\lline}{\mathbbm{l}}
\newcommand{\rline}{\mathbbm{r}}
\newcommand{\dline}{\mathbbm{d}}

\usepackage{todonotes}

\everymath{\displaystyle}
\title{3D permutations and triangle solitaire}
\author{Juliette Schabanel$^\dagger$
}
\date{$^\dagger$LaBRI, Université de Bordeaux, France}

\begin{document}

\maketitle

\begin{abstract}
    We provide a bijection between a class of $3$-dimensional pattern avoiding permutations and triangle bases, special sets of integer points arising from the theory of tilings and TEP subshifts. This answers a conjecture of Bonichon and Morel.
\end{abstract}

\section{Introduction}

In this paper, we build a bijection between objects that are seemingly unrelated: a class of pattern avoiding $3$-dimensional permutations and objects arising from the theory of subshifts called triangle bases, that was conjectured by Bonichon and Morel in \cite{BoMo22}.

Permutations are central objects in combinatorics and their study has received a lot of attention. One topic of particular interest is pattern avoidance in permutations, which led to numerous enumerative and bijective results with various objects (see \cite{Kit11, Bona} and references therein).

One can see a permutation $\sigma$ as a ($2$-dimensional) diagram with points $(i, \sigma(i))$, which satisfies the property that in each row and column lies a unique point. With this definition comes a natural generalization to higher dimensions by defining $d$-dimensional diagrams, and the notion of pattern avoidance extends naturally~\cite{AsMa10, BoMo22, PermutMach}.

The second object arises from the theory of \emph{tilings} and \emph{symbolic dynamics}. A tiling is a coloring of the grid $\Z^2$ (or of a domain of $\Z^2$) with allowed subpatterns determined by a rule set $\mathcal{R}$. 
The set of tilings respecting a certain rule set is called a \emph{subshift} \cite{Ho16}. Tilings have been widely studied during the last decades and many questions, such as the \emph{tiling problem}, which asks whether $\Z^2$ can be tiled with a given rule set, 
were proved undecidable \cite{Ber66}. However, for some classes with strong structure or combinatorial properties, such as cellular automata, some problems become decidable and a deeper study can be carried \cite{Kari05}. The class we consider in this paper, named TEP for \emph{totally extremally permutive}, was introduced by Salo in \cite{Sa22} and generalizes bipermutive cellular automata. 
The language of any TEP subshift is decidable, and they possess bases, i.e. sets of cells whose content can be chosen freely and determine the values of a larger set. Our objects of interest are the bases of triangles for the \emph{Ledrappier subshift}.

In \cite{BoMo22}, Bonichon and Morel listed classes of higher dimensional permutations avoiding small patterns and computed the first terms of the enumeration sequences these classes. It turned out that several of those sequences match existing sequences in the OEIS. For some, such as $Av(\ddincr)$ that matches the number of intervals in the weak Bruhat order, they were able to provide a bijection, but most were only conjectures. Some of these conjectures suggest that bijections for classes of pattern avoiding permutations could generalize to higher dimensions. One of these conjectures was recently proven by Bonichon, Muller and Tanasa \cite{dBaxterBMT25}, who built a bijection between $d$-Baxter permutations, which generalize Baxter permutations to higher dimensions, and high dimension floorplans. Other conjectures make links between higher dimensional pattern avoiding permutations and objects that are not known to be related to them. Of interest to us, Bonichon and Morel observed that the sets of triangle bases of size $n$ and $\A$ have the same cardinalities up to $n=8$, and they conjectured that these two sets are in bijection.
Our main result is to prove their conjecture:

\begin{thm}
    There is an explicit bijection between $3$-permutations of size $n$ avoiding patterns $\ddincr$ and $\pat$ and triangle bases of size $n$.
\end{thm}

The idea of our construction is to use the numbers inversions in the $3$-permutation to compute the coordinates of the points of the basis. The inverse bijection is obtained by converting triangle basis into \emph{fine mixed subdivisions} -or \emph{lozenge tiling}-, objects from discrete geometry which Ardila and Ceballos proved to be naturally associated to $3$-permutations avoiding $\ddincr$ \cite{AcyclicSystAC2013}. This correspondence is not injective in the general case, but we prove that when restricted to a certain class of lozenge tilings it becomes a bijection with $3$-permutations avoiding $\ddincr$ and $\pat$. 

Our bijection makes a link between objects from two previously independent classes which are hard to enumerate, and it allows us to transport structure properties and methods from one class to the other, such as a random sampling method relying on a Markov chain. In addition to that, the construction is quite simple, brings keys to understand better both objects, and could be applied to other classes of pattern avoiding $d$-permutations (even for $d>3$), maybe leading to new bijections.

\paragraph{Structure of the paper}
Precise definitions are given Section~\ref{sec:defs}. The bijection from $3$-permutations to triangle bases is described in Section~\ref{sec:gamma} and the inverse construction is built in Section~\ref{sec:psi}, relying on objects from discrete geometry known as \emph{fine mixed subdivisions} or \emph{lozenge tilings}. In Section~\ref{sec:proof}, we prove that both mappings have the announced image set and that they are inverse of each other. In Sections~\ref{sec:sums}, we describe a decomposition of bases and explain how it translates to a decomposition of $3$-permutations avoiding $\ddincr$ and $\pat$ with a notion of sums analogous to the one for separable permutations, but with a shift. In Section~\ref{sec:solitaire}, we describe a dynamical system on bases, the \emph{solitaire}, which translates to permutations though our bijection and could enable uniform sampling. Finally, in Section~\ref{sec:discu} we discuss some generalizations of the bijection as well as related open problems.

\section{Definitions and setting}
\label{sec:defs}

\subsection{3-Permutations avoiding a pattern}
\label{sec:permut}

\subsubsection{Permutations avoiding a pattern}

A permutation $\sigma = \sigma(1)\sigma(2) \ldots \sigma(n) \in \S_n$ is a bijection from $\intval{1, n} = \{1, 2, \ldots n\}$ to itself. 
The (2 dimensional) \emph{diagram} of a permutation $\sigma \in \S_n$ is the set of points $P_\sigma:=\{(1, \sigma(1)), \ldots, (n,\sigma(n))\}$ (see Figure~\ref{fig:permut ex}). The diagrams of permutations of size $n$ are exactly the point sets such that each row and column of $\intval{1,n}^2$ contains exactly one point.

A permutation $\sigma$ \emph{contains} a \emph{pattern} $\pi \in \S_k$ if there is a set of indices $i_1 < i_2 < \ldots < i_k$ such that $\sigma(i_1)\sigma(i_2)\ldots\sigma(i_k) = \pi$ (once standardized), the set $(i_1, i_2, \ldots , i_k)$ is then called an \emph{occurrence} of $\pi$ in $\sigma$. Otherwise, we say that $\sigma$ \emph{avoids} $\pi$. Given a set of patterns $\pi_1, \ldots, \pi_l$, we denote by $Av_n(\pi_1, \ldots, \pi_l)$ the set of permutations of size $n$ avoiding all the $\pi_i$s.

\begin{figure}[ht]
    \centering
    \includegraphics[page=1, height=2.5cm]{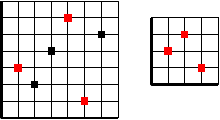}
    \caption{The permutation $\sigma = 324615$ with an occurrence of the pattern $231$ in red.}
    \label{fig:permut ex}
\end{figure}

\subsubsection{Higher dimensional permutations}

\begin{defin}
    A \emph{$d$-permutation} (or \emph{$d$-dimensional permutation}) of size $n$, is a tuple $\pmb{\sigma} = (\sigma_1, \ldots, \sigma_{d-1})$ of $d-1$ permutations of size $n$. We denote by $\S_n^{d-1}$ the set of these objects. 
    The \emph{diagram} of a $d$-permutation is the set of points $P_{\pmb{\sigma}} \coloneqq \{(i, \sigma_1(i), \ldots, \sigma_{d-1}(i)) \mid 1 \leqslant i \leqslant n\}\in\Z^d$. Note that $d$ is the dimension of the diagram.
\end{defin}
\vspace{-0.2cm}
The diagrams of $d$-permutations of size $n$ are exactly the point sets such that every hyperplane $x_i = j$ with $i \in \intval{1, d}$ and $j \in \intval{1, n}$ contains exactly one point. Examples are given in Figure~\ref{fig:forb pattern}.

\smallskip
The definition of pattern avoidance extends naturally to $d$-permutations as follows. Given $\boldsymbol{\pi} \in \S_k^{d-1}$ a pattern, a $d$-permutation $\boldsymbol{\sigma} \in \S_n^{d-1}$ \emph{contains} $\boldsymbol{\pi}$ if there is a set of indices $I \subset \intval{1, n}$ such that $\boldsymbol{\sigma}_{|I} = \boldsymbol{\pi}$ (once standardized), otherwise it \emph{avoids} it. 
Given a set of $d$-patterns $\boldsymbol{\pi}_1, \ldots, \boldsymbol{\pi_k}$ we denote by $Av_n(\boldsymbol{\pi}_1, \ldots, \boldsymbol{\pi}_k)$ the set of $d$-permutations of size $n$ that avoid all the $\boldsymbol{\pi}_i$. 

In what follows, we focus on $3$-permutations avoiding the two patterns $(\blue{12}, \red{12})$ and $(\blue{312}, \red{231})$, which are depicted in Figure~\ref{fig:forb pattern}.

\begin{figure}[ht]
    \centering
    \includegraphics[page=2, height=3cm]{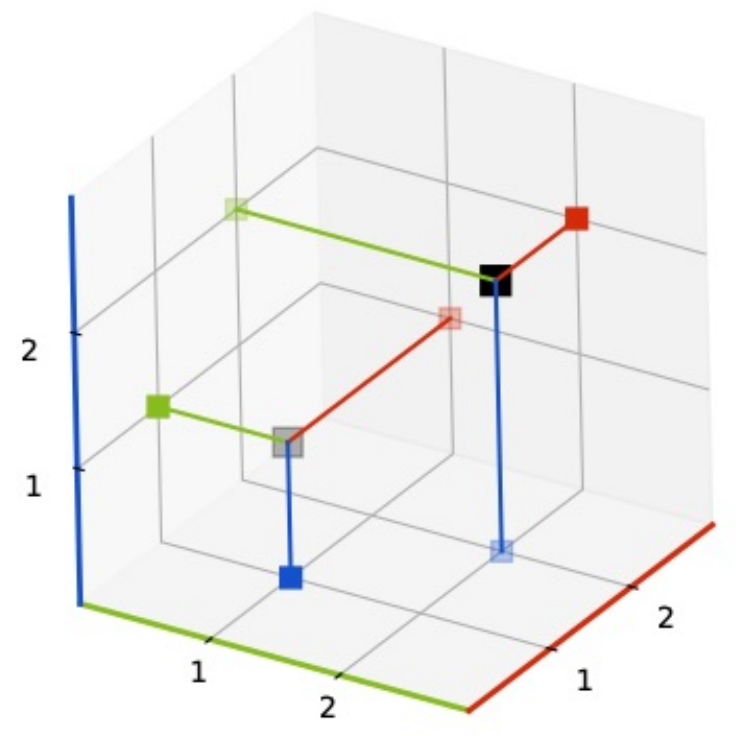}
    \quad
    \includegraphics[page=3, height=3cm]{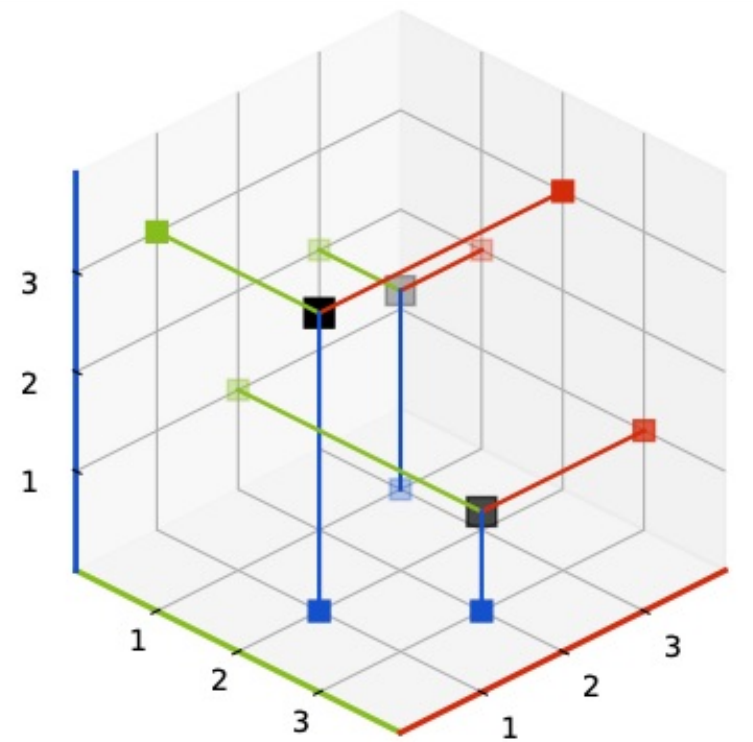}
    \quad \quad \quad \quad
    \includegraphics[page=4, height=4cm]{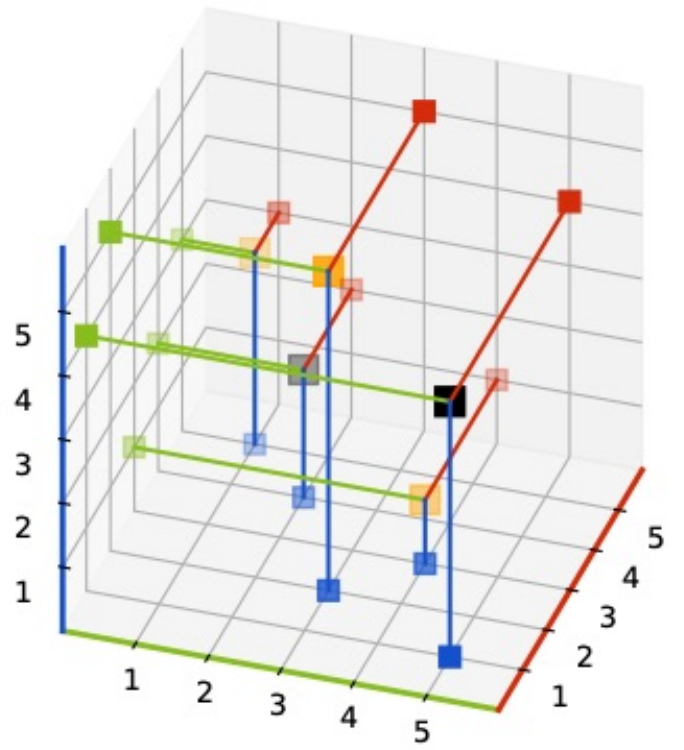}
    \caption{Left: The two forbidden patterns: $\ddincr$ (left) and $\pat$ (right). \\ Right: A $3$-permutation, $\colpair{54231}{32514}$, with an occurrence of $\pat$ at the position $(1,3,4)$, highlighted in orange.} 
    \label{fig:forb pattern}
\end{figure}

\begin{rem}
    Note that a $3$-permutation avoiding a $3$-pattern is not the same as a pair of permutations each avoiding a pattern since we require the occurrence to be on the same indices for both permutations. For example $\colpair{312}{231}$ avoids $\ddincr$ although both $312$ and $231$ contain $12$.
\end{rem}

\subsection{Tilings, configurations and bases}
\label{sec:basis}

Our other class of objects, \emph{triangle bases}, arises from the study of tilings, which we now define.

\subsubsection{Tilings}

Let $A$ be a set of symbols and $S$ be a finite subset of $\Z^2$. Let $\mathcal{R}$ be a subset of $A^S$. A \emph{tiling} with \emph{tile set} $A$ and \emph{allowed patterns} $\mathcal{R}$ is a coloring $c: D\subset\Z^2 \to A$ of a domain $D$ of the grid $\Z^2$ with symbols of $A$ such that for each translation $S'$ of $S$ we have $c_{|S'} \in \mathcal{R}$. In other words, we require that wherever one looks at the coloring through an $S$-shaped window, what one sees is in $\mathcal{R}$. The elements of $A^S \setminus \mathcal{R}$ are called the \emph{forbidden patterns} and the set of valid tilings for $\mathcal{R}$ is called the \emph{subshift} (of finite type) with rule set $\mathcal{R}$. Figure~\ref{fig:tiling ex} provides an example.

\begin{figure}[ht]
    \centering
    \includegraphics[page=1, height=2.2cm]{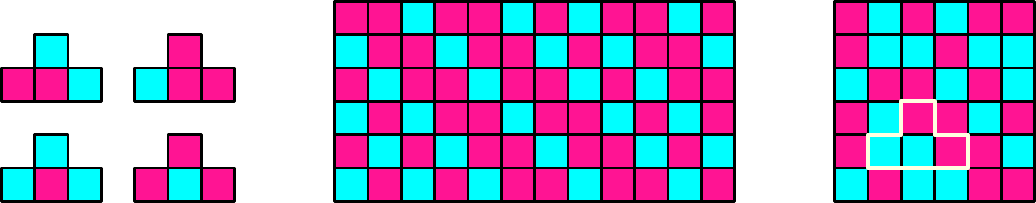}
    \caption{An example of a tiling. Left: the rule set $\mathcal{R}$. Middle: a valid tiling. \\Right: an invalid tiling with a forbidden pattern highlighted.}
    \label{fig:tiling ex}
\end{figure}

In what follows, we consider a special kind of subshifts, called \emph{TEP subshifts} (where TEP stands for Totally Extremally Permutive), which were introduced in \cite{Sa22}. A subshift is \emph{TEP} if for all $x \in S$, for each partial coloring $c: S\setminus \{x\} \to A$, there is a unique $a \in A$ such that the extension of $c$ with $c(x) = a$ is in $\mathcal{R}$. In other words, if one fills all cells of $S$ but one with arbitrary symbols of $A$, then there is a unique way to complete it into an allowed pattern. An example, the XOR automaton tiling (also known as the Ledrappier subshift), is given in Figure~\ref{fig:XOR rule}. In what follows, we only consider TEP subshifts with $S = \{(0,0), (1,0), (0,1)\}$.

\begin{figure}[ht]
    \centering
    \includegraphics[page=2, height = 2.2cm]{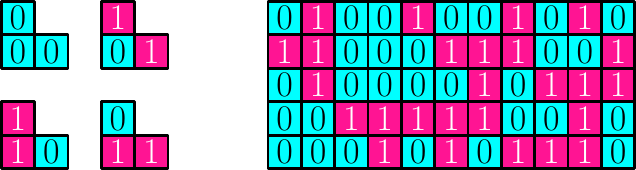}
    \caption{Left: The tiling rules of the XOR automaton, also known as the Ledrappier subshift. Right: An example of a valid tiling. Adding or removing a pattern from $\mathcal{R}$ breaks the TEP property.}
    \label{fig:XOR rule}
\end{figure}

\subsubsection{Filling}

When confronted with a TEP subshift, one natural question to ask is ``given a set of cells $P$ whose contents are known, what other values can be deduced?''. The following definition aims at formalizing this notion. 

\begin{defin}
   Let $P\subset \Z^2$ be a set of cells. Performing a \emph{filling step} of $P$ consists in choosing a position $(x,y)$ such that $|P \cap \{(x,y), (x+1,y), (x,y+1)\}| = 2$ and adding the missing point to $P$ (see Figure~\ref{fig:filling} for an example). This process is confluent and converges to a limit set, denoted $\filling(P)$, called the \emph{filling} of $P$ \cite{SaSc23}. We say that $P$ \emph{fills} if its filling is $\varphi(P) = T_{|P|}$.
\end{defin}

Intuitively, if we know the values of a tiling on $P$, then a filling step consists in using the tiling rules to deduce the value of a new cell and adding it to the known set. The filling of $P$ is then the set of values which can always be deduced from $P$.

\begin{figure}[ht]
    \centering
    \includegraphics[page=4, width=0.9\textwidth]{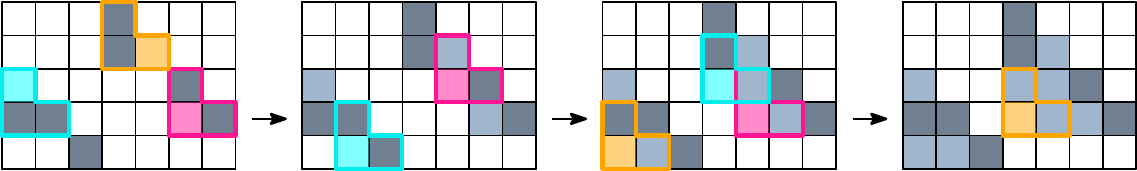}
    \caption{An example of execution of the filing process. Dark gray cells are the original set, colored ones are added by the current step and light gray cells were added earlier.}
    \label{fig:filling}
\end{figure}

A \emph{triangle} of size $k \in \N$ is a set $T = \{(x, y) \mid a \leqslant x , b \leqslant y \text{ and } x+y \leqslant a+b+k\}$ for some $a, b \in \Z$. Two sets of points $A, B \subset \Z^2$ \emph{touch} if there are points $a \in A$ and $b \in B$ such that $a=b$ or $a$ and $b$ differ by $(0,1)$, $(1,0)$ or $(1, -1)$, meaning either $a-b$ or $b-a$ is one of these vectors.

\begin{thm}[{\cite[Lemma 2]{SaSc23}}]
    For any set of cells $P$, the set $\varphi(P)$ is a union of non touching triangles whose sizes sum to at most $|P|$.
\end{thm}

\subsubsection{Independence and Bases}

The other natural question is ``given a set of cells $P$, can any choice of symbols for $P$ be extended into a valid tiling of $\Z^2$?''. When it is the case, we say that $P$ is \emph{independent}. 

A set $P$ generates a \emph{conflict} at a cell $x$ if there are two distinct minimal subsets $X$ and $Y$ of $P$ that contain $x$ in their fillings. In that case, $P$ is not independent as some choices for $X$ and $Y$ can disagree on the value at $x$.

Denote $T_n = \{(x,y) \mid x,y \in \N, x+y < n\}$ the triangle of size $n$. A set of cells $P$ is a \emph{basis} (of $T_n$) if for any partial coloring $c: P \to A$, by iteratively deducing values using the TEP rules of $\mathcal{R}$ we always end up with a valid coloring of $T_n$ (i.e. all of $T_n$ is determined by $P$ and there is no conflict). See Figure~\ref{fig:basis} for examples. We denote by $\mathcal{B}_n$ the set of bases of $T_n$. It is proved in \cite{SaSc23} that each basis of $T_n$ contains exactly $n$ points. We call \emph{configuration} of size $n$ a set $C \subset T_n$ of $n$ points.

\begin{figure}[ht]
    \centering
    \includegraphics[page=6, height = 2.5cm]{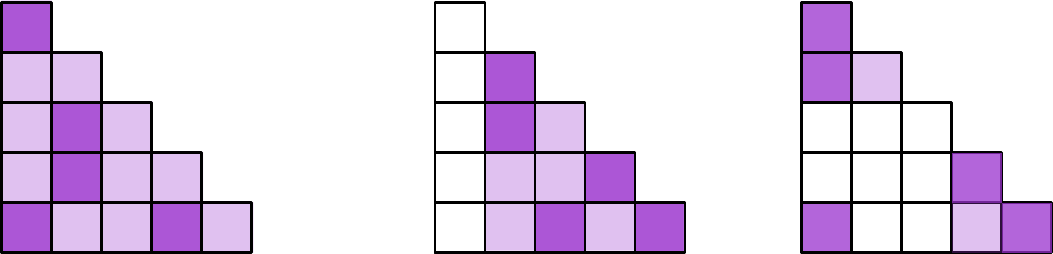}
    \caption{Configurations of size $5$ (dark cells) with their filling (light cells). Left: A basis of $T_5$. Middle: A non independent pattern. Right: A non filling pattern.}
    \label{fig:basis}
\end{figure}

\begin{thm}[{\cite[Theorem 1]{SaSc23}}]
    A configuration of size $n$ is a triangle basis if and only if its filling is $T_n$.
\end{thm}

The idea behind this result is that $n$ is the minimal amount of information needed to fill $T_n$, and if there is a conflict then one point is redundant and there is not enough left to fill all of $T_n$. 

The number of triangle bases for $n$ up to $8$ were computed in \cite{Sa22} and are the following: $1$, $3$, $16$, $122$, $1188$, $13844$, $185448$, $2781348$, $\ldots$, which is the same as the numbers of $3$-permutations avoiding $\ddincr$ and $\pat$ computed by Bonichon and Morel \cite{BoMo22}. 


\paragraph{On the choice of grid.}
\label{rem:grid}
Observe that for the filling process, the directions $(0, 1)$, $(1, 0)$ and $(1, -1)$ play symmetric roles. For this reason, it would be natural to consider configurations and bases on the triangular grid rather than the square grid. 

\begin{defin}
\label{def:grid}
    The \emph{triangular grid} is the lattice obtained by adding the integer lines of direction $(1, -1)$ to the lattice $\Z^2$ and rotating the ordinate axis of $\Z^2$ by 30°. Each cell of $\Z^2$ is then associated to an upward and a downward equilateral triangle in the triangle grid. Cells of $\Z^2$ are naturally in correspondence with upward triangle cells, endowing them with coordinates. Configurations are then sets of upward triangle cells, and downward triangle cells are ignored. 
\end{defin}
Observe that, on the triangular grid, the set $T_n$ becomes an equilateral triangle (see Figure~\ref{fig:grids}), and the symmetries of the filling rules imply that the set of basis of size $n$, $\B_n$, is stable under 120° rotation and mirror.
\begin{figure}[ht]
    \centering
    \includegraphics[page=7, height=2.5cm]{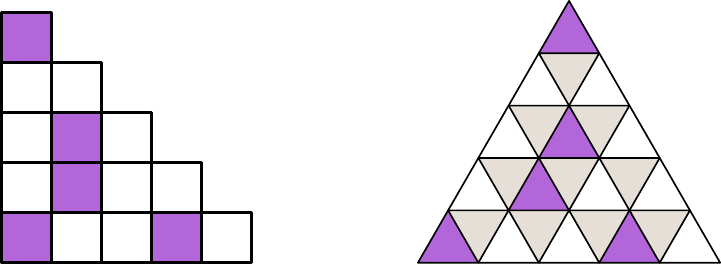}
    \caption{$T_5$ on the square and on the triangular grid, with a same basis in purple.}
    \label{fig:grids}
\end{figure}
    
The square grid is more comfortable when working with points coordinates (which we will do in Section~\ref{sec:gamma}), while the triangular grid is better to understand the symmetries of the objects and to study lozenge tilings (which we will need in Section~\ref{sec:psi}), so we will switch between the two grids depending on the context.

\section{The bijection: from 3-permutations to bases}
\label{sec:gamma}

In this section, we define a mapping that sends $3$-permutations to configurations, which induces a bijection between $\A$ and the set of bases $\B_n$, as we will prove in Section~\ref{sec:proof}. This mapping relies on the inversion sequences of the permutations, which we now define.

\subsection{Inversions}

Let $\sigma \in \S_n$ be a permutation. For $i \in \intval{1,n}$,  we denote by $r_\sigma(i)$ (resp. $\ell_\sigma(i)$) the number of $i < j \leqslant n$ (resp. $1 \leqslant j <i$) such that $\sigma(i)>\sigma(j)$ (resp. $\sigma(i)<\sigma(j)$). These are called the number of \emph{right} (resp. \emph{left}) \emph{inversions} of $\sigma$ at $i$ and the sequence $(r_\sigma(i))_{1 \leqslant i \leqslant n}$ (resp. $(l_\sigma(i))_{1 \leqslant i \leqslant n}$) the \emph{right} (resp. \emph{left}) \emph{inversion sequence} of $\sigma$.

Looking at the diagram of $\sigma$, for each $i$ the number $r_\sigma(i)$ is the number of points to the bottom right of $(i, \sigma(i))$ and $\ell_\sigma(i)$ is the number of points to the top left of it (see Figure~\ref{fig:inv num}). We denote respectively by $R_\sigma(i)$ and $L_\sigma(i)$ those sets of points. 

\begin{figure}[ht]
    \centering
    \includegraphics[page=1, height = 3cm]{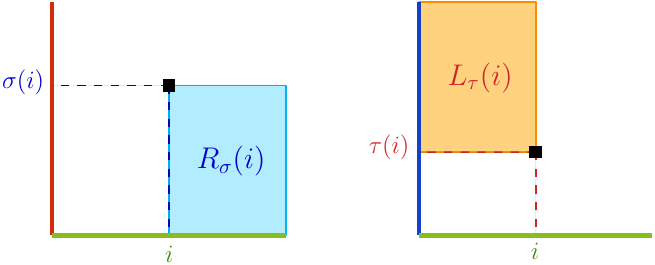}
    \caption{The set $R_\sigma(i)$ (resp. $L_\tau(i)$) is the set of points in the blue (resp. orange) area.}
    \label{fig:inv num}
\end{figure}

\begin{prop}
    The knowledge of $\sigma$ is equivalent to the knowledge of one of its inversion sequence.
\end{prop}
\begin{proof}
    Observe that $\sigma(1) = r_\sigma(1) + 1$ and if $\sigma'$ is the permutation obtained after deleting $1$, then $r_{\sigma'} (i) = r_\sigma(i+1)$ for all $1 \leqslant i < n$. This allows to recover all of $\sigma$ by inserting the points in increasing order.
\end{proof}

\subsection{The function \texorpdfstring{$\Gamma$}{Gamma}: from 3-permutations to configurations}

Let $\Gamma$ be the function that maps a $3$-permutation $\sigtau \in \S_n^2$ to the set of the points $\Gamma(\sigtau) \coloneqq \{(r_\sigma(i), l_\tau(i)) \mid  1 \leqslant i \leqslant n\}\subset T_n$. When there is no ambiguity, we write $x_i = \x(i)$, $y_i = \y(i)$ and denote by $p_i$ the point $(x_i, y_i) \in \Gamma(\sigtau)$. 

\begin{example}
    Consider the $3$-permutation (\blue{254361}, \red{624315}). Its inversion sequences are (\blue{132110}, \red{011241}) and its image through $\Gamma$ is depicted in Figure~\ref{fig:Gamma ex}.
\begin{figure}[ht]
    \centering
    \includegraphics[page=2, height = 3cm]{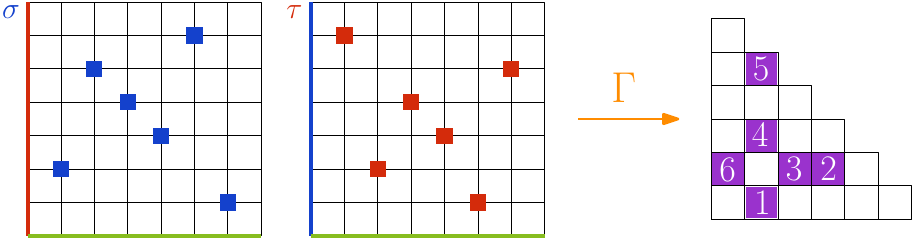}
    \caption{The diagram of the $3$-permutation (\blue{254361}, \red{624315}) and its image through $\Gamma$. Labels are not part of the image, they only indicate which integer is used to compute the coordinates of the point.}
    \label{fig:Gamma ex}
\end{figure}
\end{example}

\begin{prop}
\label{prop:image size}
    For any $\sigtau \in Av_n(\ddincr)$, the image $\Gamma(\sigtau)$ is a configuration of size $n$.
\end{prop}
\begin{proof}
    Observe that for any $3$-permutation $\sigtau \in \S_n^2$, for each $i \in \intval{1,n}$, we have $\x(i) \leqslant n-i$ and $\y(i) < i$ so $\Gamma(\sigtau) \subset T_n$. Moreover, if $\sigtau \in \S_n^2$ avoids $\ddincr$ then for all $i<j$ either $\sigma(i) > \sigma(j)$ and so $\x(i) > \x(j)$ or $\tau(i) > \tau(j)$ and so $\y(i) < \y(j)$, hence all $p_i$'s are distinct.
\end{proof}

\begin{rem}
    Note that $\Gamma$ is not injective on $Av_n(\ddincr)$ for $n >3$: for instance the $3$-permutations $\colpair{4132}{2431}$ and $\colpair{4213}{3241}$ have the same image and both avoid $\ddincr$.

    One can define a variant of $\Gamma$, denote it by $\widetilde{\Gamma}$, that sends a $3$-permutation avoiding $\ddincr$ to a \emph{labeled} configuration by labeling each point with the integer used to compute its coordinates. Observe that $\widetilde{\Gamma}$ is easily invertible: given the labels of the points, we can recover the inversion sequence and so the permutation. However, our triangle bases are not labeled so we have to forget this labeling and find a way to recover it (the condition ``if $i < j$ then either $\x(i) > \x (j)$ or $\y(i) < \y(j)$'' only gives a partial order). This will be the object of Section~\ref{sec:psi}. Still, the labeled variant might be useful for some generalizations (see Section~\ref{sub:other classes}).
\end{rem}

\begin{rem}
    Observe that $\Gamma$ carries some symmetries of both objects: rotating the diagram of $\sigtau$ around the axis $x+y+z$ (i.e. applying $\rho: \sigtau \mapsto (\red{\tau^{-1}}, \green{\sigma \circ \tau^{-1}})$, which stabilizes $\A$) rotates the image by 120° and, denoting $\desc \coloneqq n(n-1) \ldots 21$, applying the mirror $\mu:\sigtau \mapsto (\red{\desc \circ \tau \circ \desc}, \blue{\desc \circ \sigma \circ \desc})$ exchanges the abscissa and ordinates axes (see Figure~\ref{fig:sym}).
    \begin{figure}[ht]
    \centering
    \includegraphics[page=6, width =\textwidth]{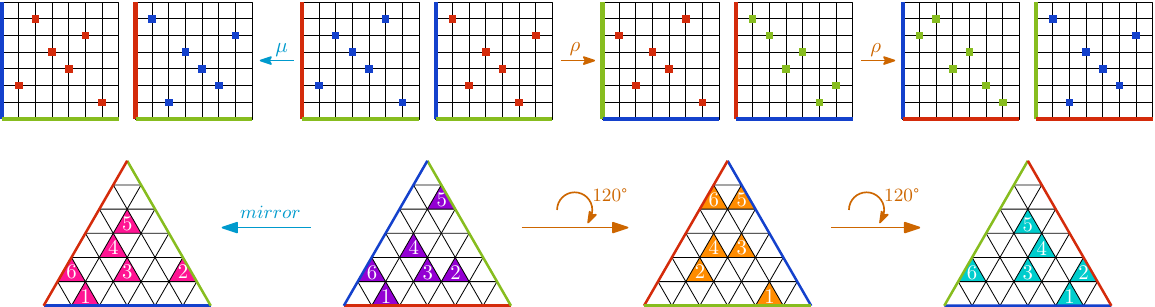}
    \caption{The function $\Gamma$ transports rotational and mirror symmetries.}
    \label{fig:sym}
    \end{figure}
    
    It may seem that our construction $\Gamma$ breaks the rotational symmetry by forgetting one of the three directions, but this is not the case. Observe that a point $(x, y) \in T_n$ naturally has a third redundant coordinate: it is at distance $x$ from the vertical side of $T_n$, $y$ from the horizontal side and $z=n-1-x-y$ from the diagonal side. These $3$ coordinates are exchanged by the 120° rotation, and $\Gamma$ obtains them by partitioning the set of $n-1$ points in the diagram of the $3$-permutation. Indeed, consider $\sigtau \in Av_n(\ddincr)$, $i \in \intval{1,n}$ and the three hyperplanes that go through the point $i$ in the diagram of $\sigtau$ which cut the cube into $8$ octants. The $(1,1,1)$ and $(n,n,n)$ octants are empty as $\sigtau$ avoids $\ddincr$, and then $R_\sigma(i)$ is the octants $(n, 1, 1)$ and $(n, 1, n)$, and $L_\tau(i)$ the octants $(1, 1, n)$ and $(1, n, n)$. This partitions the cube into three boxes, which are exchanged by a rotation of the diagram and the coordinates of the $p_i$ is $\Gamma(\sigtau)$ are the number of points these boxes contain (see Figure~\ref{fig:box}).
    \begin{figure}[ht]
        \centering
        \includegraphics[page=7, height=4cm]{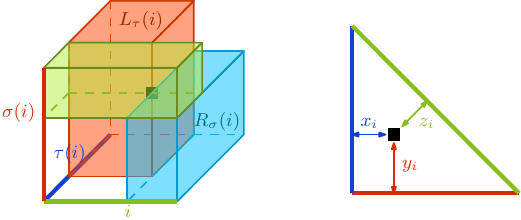}
        \caption{How $\Gamma$ partitions the cube to obtain the three coordinates of a point.}
        \label{fig:box}
    \end{figure}
\end{rem}

At this point, it is not clear that the image of a $3$-permutation avoiding $\ddincr$ and $\pat$ is indeed a triangle basis. Intuitively, avoiding $\ddincr$ will ensure that the points are ``not too close'' (independence) and avoiding $\pat$ that they are ``not too far'' (filling), see Figure~\ref{fig:pat inv}. In the next subsection, we formalize the first notion and prove that the image of a $3$-permutation avoiding $\ddincr$ does satisfy this condition. The proof that $\Gamma(\A)\subseteq \B_n$ will be delayed to Section~\ref{sec:proof}, as it will be easier with the inverse bijection.
\begin{figure}[ht]
    \centering
    \includegraphics[page=3, height = 2cm]{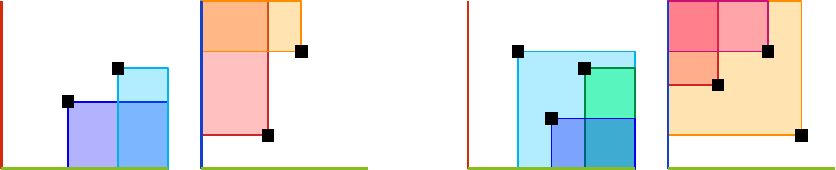}
    \caption{The positions forbidden by the patterns $\ddincr$ (left) and $\pat$ (right).}
    \label{fig:pat inv}
\end{figure}

\subsection{Sparsity}

In this subsection, we provide a necessary condition for independence and prove that if a $3$-permutation avoids $\ddincr$ then its image through $\Gamma$ satisfies this condition.

For integers $a, b$ and $k$, we denote by $(a,b) + T_k$ the triangle $\{(x, y) \mid a \leqslant x, b \leqslant y$ and ${x + y < a+b+k}\}$.  A configuration $C$ of size $n$ is \emph{sparse} if for all $1 \leqslant k < n$, there is no triangle $T = (a,b) + T_k \subset T_n$ such that $|C \cap T| > k$.

\begin{prop}[{\cite[Lemma 3]{SaSc23}}]
    All independent sets (and so all bases) are sparse.
\end{prop}

Intuitively, if a configuration is not sparse then we have too much information in one area and it creates a conflict. The interested reader is referred to \cite{SolitIndep[SaSc25]} for discussions on the subtle difference between sparsity and independence. 

\begin{thm}
\label{thm:sparse}
    If $\sigtau \in \S_n^2$ avoids $\ddincr$ then $\Gamma(\sigtau)$ is sparse.
\end{thm}


\begin{proof}
    Let $\sigtau \in Av_n(\ddincr)$. Assume that there is a triangle $T = (a,b)+T_k \subset T_n$ of size $k$ such that $|\Gamma(\sigtau) \cap T|>k$. Denote by $I \subset \intval{1, n}$ the set of indices $i$ such that $p_i \in T$. Recall that by definition, this is the set of indices such that $\x(i) \geqslant a$, $\y(i) \geqslant b$ and $\x(i) + \y(i) <a+b+k$. 
    
    Denote $i=\min I$, $j = \max I$ and let $m \in I$ be the point where $\sigma$ reaches its maximum on $I$ (possibly $m=i$ or $m=j$). Observe that, by definition, $L_\tau(i)\cap I = R_\sigma(j) \cap I = \varnothing$ and, since $\sigtau$ avoids $\ddincr$, we have $(R_\sigma(m) \cup L_\tau(m)) \cap I = I\setminus\{m\}$. Therefore we have 
    \vspace{-0.3cm}
    \begin{multline*}
        x_m + y_m = |R_\sigma(m)| + |L_\tau(m)| \geqslant |(R_\sigma(m) \cup L_\tau(m)) \cap I| + |R_\sigma(m)\setminus I| + |L_\tau(m)\setminus I| \\
        \geqslant |I|-1+|R_\sigma(j)|+|L_\tau(i)|\geqslant k+a+b
    \end{multline*}
    which contradicts $p_m\in T$.
\end{proof}

\section{The inverse bijection: from bases to 3-permutations}
\label{sec:psi}

In this section, we build the inverse bijection, $\Psi$, by adapting a construction from discrete geometry. This construction is defined on objects called \emph{fine mixed subdivisions of a simplex}, which are also known as lozenge tilings of the triangle (see for instance \cite{FlagAB07}).

\subsection{Lozenge tilings}

Let us start by defining lozenge tilings and explaining how they are related to triangle bases. 

For the purposes of this section, it will be more comfortable to work on the triangular grid. Recall that on this grid $T_n$ is an equilateral triangle divided into $n(n+1)/2$ upward triangle cells and $n(n-1)/2$ downward triangle cells (see Definition~\ref{def:grid}).

\begin{defin}
    Let $C$ be a configuration of $n$ upward triangles in $T_n$. A \emph{lozenge tiling} of $C$ is a tiling of $T_n \setminus C$ by $n(n-1)/2$ rhombi formed by combining two unit equilateral triangles edge-to-edge (see Figure~\ref{fig:lozenge}). The size of a lozenge tiling is its number of triangles.
\end{defin}

\begin{figure}[ht]
    \centering
    \includegraphics[page=1, height=3.5cm]{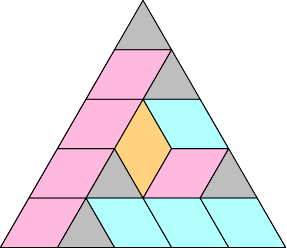}
    \caption{A lozenge tiling of size $5$.}
    \label{fig:lozenge}
\end{figure}

We proceed by associating a lozenge tiling to a basis, and then applying a construction that associates a $3$-permutation to a lozenge tiling. The following result ensures this is possible.

\begin{prop}[{\cite[Theorem 6.2]{FlagAB07}}]
    A triangle configuration admits a lozenge tiling if and only if it is sparse. 
\end{prop}
\begin{proof}
    It is easy to see that sparsity is a necessary condition to admit a lozenge tiling. Indeed, if $T$ is a subtriangle of $T_n$, then all the downward unit triangles of $T$ must be covered by a rhombus that is completely inside $T$. Denoting $k=|T|$, it follows that there must be at least $k(k-1)/2$ rhombi in $T$, so there cannot be more than $k$ triangles in $T$. We refer to \cite{FlagAB07} for a proof of the sufficiency. 
\end{proof}

In fact, for bases, we have the following stronger result.

\begin{lemma}[{\cite[Corollary 3.2, rephrased]{MixedSubdivYY2026}}]
    A configuration admits a unique lozenge tiling if and only if it is a basis.
\end{lemma}
\begin{proof}
    Let $B$ be a triangle basis. Observe that when two cells of $B$ are in position for a filling step, then a lozenge tiling of $B$ must contain the rhombus that completes the triangle of size $2$ containing those two cells (see Figure~\ref{fig:lozenge unique}). Since $B$ fills $T_{|B|}$, this implies that the whole tiling is entirely determined. In fact, this situation is the only one that can force the position of a rhombus. Therefore, if a configuration admits a unique tiling, then it fills, so is a basis. 
    \begin{figure}[ht]
        \centering
        \includegraphics[page=2,height=1.5cm]{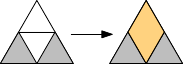}
        \caption{If a filling step can be performed, then a rhombus must be placed there.}
        \label{fig:lozenge unique}
    \end{figure}
\end{proof}

Given a basis $B$, we then denote by $\mathcal{L}_B$ the unique lozenge tiling of $B$.

\subsection{The function \texorpdfstring{$\Psi$}{Psi}: from bases to 3-permutations}

We now describe a construction from discrete geometry which associates a pair of permutations to a lozenge tiling. This construction can be found for instance in \cite{AcyclicSystAC2013}. 

Let $\mathcal{L}$ be a lozenge tiling of size $n$. We build its $3$-permutation, $\widehat{\Psi}(\mathcal{L})$, as follows (see Figure~\ref{fig:ex psi} for an example):
\begin{enumerate}
\itemsep=0pt
    \item Number the segments of the right side of $T_n$ from $1$ to $n$ from bottom to top. 
    \item For each $i \in \intval{1,n}$, starting in the middle of the $i$th segment of the right side of $T_n$, draw a line $\Yline(i)$ as follows:
    \vspace{-0.2cm}
    \begin{itemize}
    \item[-] When the line enters a rhombus, it crosses it transversally. 
        \item[-] When the line enters a triangle, it splits in two and one line exists from each of the two other sides of the triangle. 
        \item[-] When the line reaches a side of $T_n$, it stops and the number $i$ is written on that segment.
    \end{itemize}
    \vspace{-0.15cm}
    \item Denote by $\sigma$ the inverse of the permutation obtained by reading the numbers on the bottom side of $T_n$ from left to right and denote by $\tau$ the inverse of the permutation obtained by reading the numbers on the left side from top to bottom. Return $\widehat{\Psi}(\mathcal{L}) \coloneqq \sigtau$.

    In other words, if a line goes from the $i$th segment of the right side of $T_n$ to the $j$th segment of the bottom side (numbering from left to right), then $\sigma(i) = j$. Similarly if a line goes from the $i$th segment of the right side to the $j$th segment of the left side (numbering from top to bottom) then $\tau(i) = j$.
\end{enumerate}

\begin{figure}[ht]
    \centering
    \includegraphics[page=3, height=5.5cm]{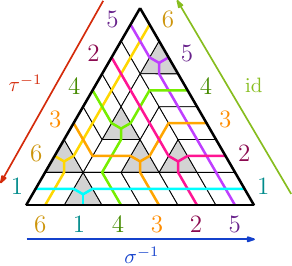}
    \caption{An example of application of $\Psi$ to a basis of size $6$. The basis $B$ is made of the gray triangles and $\Psi(B) = \colpair{254361}{624315}$.}
    \label{fig:ex psi}
\end{figure}

\begin{rem}
    A lozenge tiling of a triangle configuration is often referred to as a \emph{fine mixed subdivision} of $2$-dimensional simplex (i.e. of the triangle) in discrete geometry. The set of lines appearing in the construction of $\Psibis$ can be referred to in the literature as the \emph{colored dual polyhedral complex} or the \emph{tropical pseudoline arrangement} of the fine mixed subdivision \cite{FlagAB07, AcyclicSystAC2013}.
\end{rem}

We now provide an elementary proof that $\Psibis$ is well defined.

\begin{lemma}
\label{lem:psi well def}
    The construction $\Psibis$ is well defined on any lozenge tiling.
\end{lemma}
\begin{proof}
    Let $\loz$ be a lozenge tiling of size $n$ and consider the lines $\Yline(i)$ drawn when building $\Psibis(\loz)$. Observe that, since lines cross rhombi transversely, a line can only cross left slanted edges until it reaches a unit triangle, and crosses them from right to left. Then, the two lines it splits into can respectively only cross horizontal edges from top to bottom and right slanted edges from right to left. It follows that those two lines can never enter another triangle, and must reach respectively the bottom and the left border of $T_n$. Finally, since the path followed by the lines are reversible, no two lines can reach the same segment on a border. Therefore, the words written on the bottom and left sides of $T_n$ are well defined and are permutations in $\S_n$
\end{proof}

Finally, we define the mapping $\Psi$ on bases as the function that maps a basis $B$ to the $3$-permutation $\widehat{\Psi}(\mathcal{L}_B)$ associated to its unique lozenge tiling $\loz_B$.

\subsection{Acyclicity}

Ardilla and Ceballos characterized the $3$-permutations obtainable as the image of a lozenge tiling. This result is known as the ``2D Acyclic System Theorem''.

\begin{thm}[\cite{AcyclicSystAC2013}, Theorem 4.2, rephrased]
\label{thm:image psi}
    The set of $3$-permutations obtainable as the image of a lozenge tiling through $\Psibis$ is exactly $Av(\ddincr)$.
\end{thm}
\begin{proof}
    We reprove that the image of a lozenge tiling is a $3$-permutation avoiding $\ddincr$ as it provides insight on how the lozenge tiling and the $3$-permutation are related, which will be useful later. We refer to \cite{AcyclicSystAC2013} for a proof of surjectivity. 
    
    Let $\mathcal{L}$ be a lozenge tiling of size $n$ and denote by $\sigtau$ its image through $\Psibis$. For $i \in \intval{1,n}$, we consider the set of $3$ lines $\Yline(i)$ from the construction of $\Psibis(\mathcal{L})$, which go from a unit triangle, $t_i$, to a border of $T_n$. We call respectively right, left and downward lines $i$ the lines exiting the triangle $t_i$ to its right, left or bottom side and denote them respectively $\rline_i$, $\lline_i$ and $\dline_i$. In Figure~\ref{fig:psi_av_1212}, the right lines are green, the left lines are blue and the downward lines are red.

    Observe that for any fixed $i\neq j \in \intval{1,n}$, $\Yline(i)$ and $\Yline(j)$ must intersect at least once for all lines to reach their border. Moreover, there is exactly one intersection of two lines in each rhombus of the tiling, and there are $n(n-1)/2$ rhombi which is exactly the number of pairs of indices. So for any $i\neq j \in \intval{1,n}$, $\Yline(i)$ and $\Yline(j)$ intersect exactly once. The result then follows from the fact that a $\ddincr$ patterns requires at least two crossings.  
    
    Let us be more precise. Fix $i < j \in \intval{1, n}$. As observed in the proof of Lemma~\ref{lem:psi well def}, a line can only cross edges that have a fixed direction among the three possible directions. It follows that $\Yline(i)$ and $\Yline(j)$ can only cross in one of the three ways depicted in Figure~\ref{fig:psi_av_1212}.

    \begin{figure}[ht]
        \centering
        \includegraphics[page=4, width=0.9\textwidth]{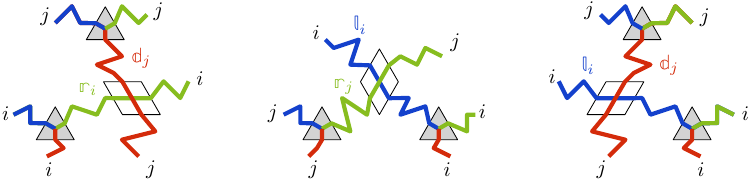}
        \caption{The 3 kinds of crossings. \\ 
        Left: Right-down crossing. Then $(i, j)$ is an inversion for $\tau$ only. \\
        Middle: Left-right crossing. Then $(i, j)$ is an inversion for $\sigma$ only. \\
        Right: Down-left crossing. Then $(i, j)$ is an inversion for both $\sigma$ and $\tau$.}
        \label{fig:psi_av_1212}
    \end{figure}

    Assume that the right line $\rline_i$ crosses the downward line $\dline_j$ (Figure~\ref{fig:psi_av_1212}, left). Then the downward line $\dline_i$ is always to the left of the downward line $\dline_j$, which means $\sigma(i) < \sigma(j)$; and the left and right lines $\lline_i$ and $\rline_i$ are always below the left and right lines $\lline_j$ and $\rline_j$, which means $\tau(i) > \tau(j)$. So $(i, j)$ is an inversion for $\tau$ but not for $\sigma$. Similarly, if the left line $\lline_i$ crosses the right line $\rline_j$ (Figure~\ref{fig:psi_av_1212}, middle), then $\sigma(i) > \sigma(j)$ and $\tau(i) < \tau(j)$ so $(i, j)$ is an inversion for $\sigma$ but not for $\tau$. Finally, assume that the left line $\lline_i$ crosses the downward line $\dline_j$ (Figure~\ref{fig:psi_av_1212}, right). Then the downward line $\dline_i$ is always to the right of the downward line $\dline_j$ so $\sigma(i) > \sigma(j)$, and the left line $\lline_i$ is always below the left line $\lline_j$ so $\tau(i) > \tau(j)$. Therefore $(i,j)$ is an inversion for both $\sigma$ and $\tau$. In all cases, $(i, j)$ is not an occurrence of $\ddincr$, so, as $i$ and $j$ are arbitrary, $\Psibis(\loz)$ avoids $\ddincr$.
\end{proof}

\begin{rem}
    Observe that $\Psibis$ is not injective. For example, the $3$-permutation $\colpair{312}{231}$ has two preimages (see Figure~\ref{fig:psi_not_injective}).
    \begin{figure}[ht]
        \centering
        \includegraphics[page=5, height=2.7cm]{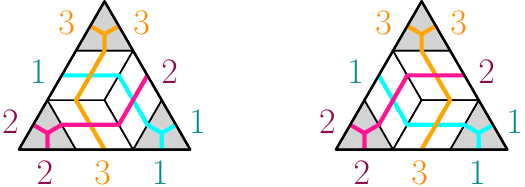}
        \caption{The $3$-permutation $\colpair{312}{231}$ has two preimages through $\Psibis$.}
        \label{fig:psi_not_injective}
    \end{figure}
\end{rem}

\section{The mappings \texorpdfstring{$\Gamma$}{Gamma} and \texorpdfstring{$\Psi$}{Psi} are inverse bijections}
\label{sec:proof}

In this section, we prove that $\Gamma$ and $\Psi$ are bijections between $\A$ and the set $\B_n$ of triangle bases of size $n$ for all $n$, and that they are inverse of each other. 

\begin{thm}
\label{thm:gamma and psi}
    The mappings $\Gamma$ and $\Psi$ are inverse of each other and induce a bijection between the set of triangle bases of size $n$ and $\A$ for all $n \geqslant 1$. 
\end{thm}

\subsection{\texorpdfstring{$\Gamma$}{Gamma} is the inverse of \texorpdfstring{$\Psi$}{Psi}}

The goal of this subsection is to prove that $\Gamma\circ\Psi$ is the identity on bases. In fact, we prove the stronger result that if a $3$-permutation $\sigtau$ is the image of a lozenge tiling $\mathcal{L}$, then the positions of the triangles in $\mathcal{L}$ are given by $\Gamma(\sigtau)$. Ardila and Ceballos describe another construction that allows to recover the positions of the triangles from the $3$-permutation, relying on a factorization of the permutations into cycles of a particular form \cite[Lemma 4.7]{AcyclicSystAC2013}. It can be proven that their construction yields the same coordinates as $\Gamma$, but we prefer proving directly that $\Gamma$ recovers the coordinates of the triangles as it is both simpler and more enlightening. 

\begin{thm}
\label{thm:psi coord}
    For any lozenge tiling $\mathcal{L}$, the coordinates of the triangles of $\mathcal{L}$ are given by $\Gamma(\Psibis(\mathcal{L}))$. 
\end{thm}
\begin{proof}
    Let $\mathcal{L}$ be a lozenge tiling of size $n$ and denote by $\sigtau$ its image through $\Psibis$. Fix $i \in \intval{1, n}$ and consider, with the same notations as in the proof of Theorem~\ref{thm:image psi}, the set of lines $\Yline(i) = \{\rline_i, \lline_i, \dline_i\}$ and the triangle $t_i$ it crosses. Let $(x_i, y_i)$ be the coordinates of the triangle $t_i$. Then observe that the abscissa $x_i$ is exactly the number of rhombi the left line $\lline_i$ crosses, as each time the line $\lline_i$ crosses a rhombus its abscissa decreases by one. According to the proof of Theorem~\ref{thm:image psi}, there is a correspondence between those rhombi and the right inversion set of $i$ for $\sigma$. Therefore we do have $x_i = \x(i)$. A symmetric argument gives that the ordinate of $t_i$ is $y_i=\y(i)$. 
\end{proof}

It remains to prove that the image of bases through $\Psi$ is exactly $\A$. To do so, we rely on the notion of cycle flips introduced by Yao and Yudin \cite{MixedSubdivYY2026}.

\subsection{Cycle-flips}

In \cite{MixedSubdivYY2026}, Yao and Yudin observed that pairs of lozenge tilings on the same triangle configuration are related by so called cycle-flips. In this subsection, we recall this notion which will help us in proving that the image of bases is exactly $\A$.   

\begin{defin}
    Recall that a rhombus can be seen as the union of an upward and a downward unit equilateral triangles, and therefore covers an upward and a downward triangle cell of the triangular grid in a lozenge tiling. A \emph{cycle} in a lozenge tiling is a sequence $\mathrm{r}_1, \ldots, \mathrm{r}_k$ of rhombi such that for each $1\leqslant i\leqslant k$, the upward triangle of $\mathrm{r}_i$ is adjacent to the downward triangle of $\mathrm{r}_{i+1}$ (with $\mathrm{r}_{k+1}=\mathrm{r}_1$). Observe that the portion of the grid covered by a cycle is an alternation of upward and downward triangle cells, which can be tiled in two ways (corresponding to the two ways to match adjacent triangles to form rhombi). 
    A \emph{cycle flip} consists in replacing one of the two tilings of a cycle by the other (see Figure~\ref{fig:cycle-flip}).
    \begin{figure}[ht]
        \centering
        \includegraphics[page=2, height=4cm]{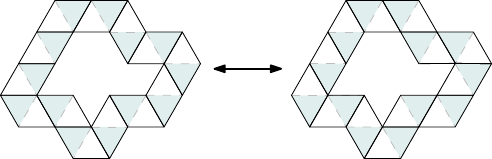}
        \caption{An example of a cycle flip.}
        \label{fig:cycle-flip}
    \end{figure}
\end{defin}

Note that if a lozenge tiling contains a cycle, then flipping this cycle yields a different lozenge tiling on the same triangle configuration. In particular, this configuration cannot be a basis. This will be the key to proving that the image of a basis through $\Psi$ avoids $\pat$. However, this will not suffice to prove the equivalence as it is hard to understand what does the presence of a cycle in a lozenge tilings implies for its image through $\Psibis$ in terms of patterns occurrences (see Section~\ref{sub:GD-flips} for further discussions on this matter). Fortunately, Yao and Yudin proved that we only need to consider cycles of the following simple shape: 

\begin{defin}
    Fix $m \geqslant 0$ and consider the region obtained by removing the three corner unit triangles of $T_{m+3}$, as well as the triangle of size $m$ in the middle. There are exactly two ways to tile this region with rhombi, as illustrated in Figure~\ref{fig:GD-flip}, 
    both leading to a cycle. Such a cycle is called a \emph{GD-cycle}. A \emph{GD-flip} consists in flipping a GD-cycle, i.e. replacing one of the two ways of tiling the region by the other. The size of a GD-cycle or GD-flip is the size of its middle triangle.
    \begin{figure}[ht]
        \centering
        \includegraphics[page=3, height=3.2cm]{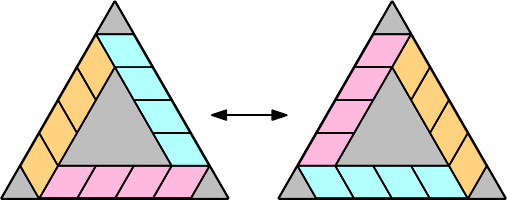}
        \caption{A GD-cycle of size $3$. A GD-flip consists in replacing one pattern with the other.}
        \label{fig:GD-flip}
    \end{figure}
\end{defin}

\begin{thm}[{\cite[Theorem 5.2]{MixedSubdivYY2026}}]
\label{thm:GD-flips}
    Any two lozenge tilings with the same configuration of triangles are connected by a sequence of GD-flips.
\end{thm}

\begin{corol}
\label{cor:basis GD}
    A lozenge tiling has no GD-flip if and only if its triangle configuration is a basis.
\end{corol}

\subsection{The image of bases is exactly \texorpdfstring{$\A$}{Av((12,12),(312,231))}}

Let us now prove that the image through $\Psi$ of the set of triangle bases $\B_n$ is exactly $\A$.

\begin{lemma}
\label{lem:image psi}
    Let $C$ be a configuration and let $\loz$ be a lozenge tiling of $C$. Then $\Psibis(\loz)$ avoids $\pat$ if and only if $C$ is a basis.
\end{lemma}
\begin{proof}
    Let us prove the contraposition. First consider a configuration $C$ that is not a basis, and a lozenge tiling $\loz$ of it. Then by Corollary~\ref{cor:basis GD}, $\loz$ contains a GD-cycle. Let $i <j < k$ be the indices of the lines that cross the GD-cycle in the construction of $\Psibis(\loz)$ as illustrated in Figure~\ref{fig:GD pat}. Then the left line $\lline_i$, the right line $\rline_j$ and the downward line $\dline_k$ cross each others, which implies that $(i,j,k)$ is an occurrence of $\pat$ in $\Psibis(\loz)$. 
    \begin{figure}[ht]
        \centering
        \includegraphics[page=4, height=4.5cm]{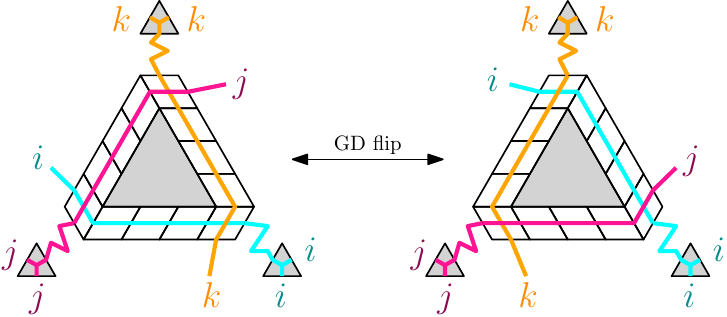}
        \caption{A GD-cycle yields an occurrence of $\pat$.}
        \label{fig:GD pat}
    \end{figure}
    \smallskip

    For the converse, let $\sigtau$ be a $3$-permutation avoiding $\ddincr$ and containing an occurrence $(i, j, k)$ of $\pat$. Let $\loz$ be a lozenge tiling such that $\Psibis(\loz) = \sigtau$ (Theorem~\ref{thm:image psi} ensures it exists). Consider the three sets of lines $\Yline(i)$, $\Yline(j)$ and $\Yline(k)$. Since $(i,j,k)$ is an occurrence of $\pat$, they must reach the bottom border in the order $(j,k,i)$ and the left one in the order $(k,i,j)$. Combining this with the fact that two sets of lines intersect exactly once, we obtain that $\Yline(i)$, $\Yline(j)$ and $\Yline(k)$ must intersect in one of the two ways depicted in Figure~\ref{fig:pat cycle}. 
    \begin{figure}[ht]
        \centering
        \includegraphics[page=5, height=5cm]{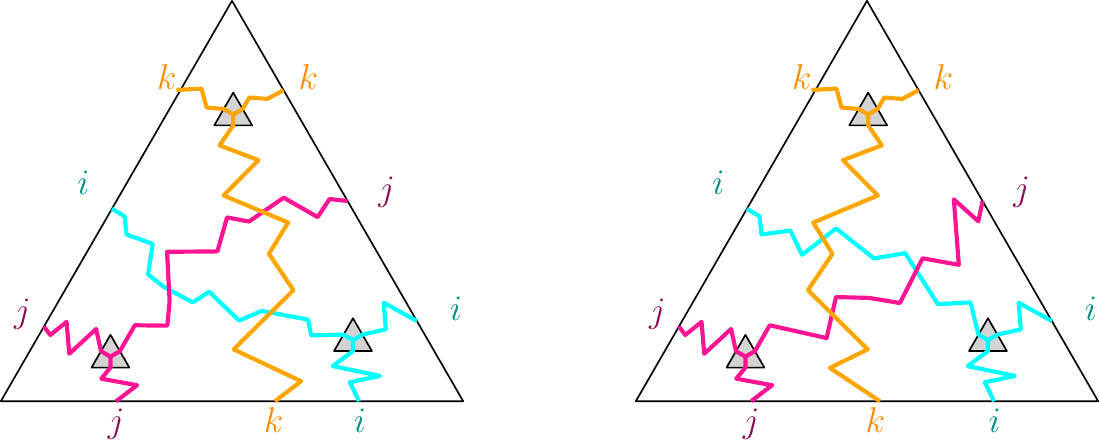}
        \caption{How the lines corresponding to an occurrence of $\pat$ can intersect. }
        \label{fig:pat cycle}
    \end{figure}

    Now observe that each line in $\Yline(i)$, 
    followed from the unit triangle $t_i$ to the border, crosses a sequence of rhombi such that the upward triangle of one is incident to the downward triangle of the next. It follows that there is a cycle in 
    $\loz$ formed by a part of the rhombi crossed by $\lline_i, \rline_j$ and $\dline_k$. In particular, flipping this cycle yields a different tiling on the same configuration of triangles, so this configuration is not a basis. Therefore, $\sigtau$ is not the image through $\Psi$ of a basis.
\end{proof}

\begin{rem}
    The use of GD-cycles in the proof was not necessary: the arguments for the presence of GD-cycle yielding an occurrence of $\pat$  hold for any cycle that is crossed by three lines. However, cycles can be crossed by an arbitrary large (odd) number of lines, so we still need a result stating that if there is a cycle in a lozenge tiling then there is a cycle with only three lines. 
\end{rem}

We can now prove Theorem~\ref{thm:gamma and psi}.

\begin{proof}[Proof of Theorem~\ref{thm:gamma and psi}]
    Recall that according to Theorem~\ref{thm:image psi}, every $3$-permutation avoiding $\ddincr$ can be obtained as the image of a lozenge tiling. Then, Lemma~\ref{lem:image psi} states that a $3$-permutation avoiding $\pat$ can only be obtained as the image of a tiling of a basis, therefore the image $\Psi(\B_n)$ of bases of size $n$ is exactly $\A$. Finally, Theorem~\ref{thm:psi coord} states that $\Gamma\circ\Psi$ is the identity on $\B_n$ for any $n$, thus $\Gamma$ and $\Psi$ are bijections between $\B_n$ and $\A$ for all $n$, and they are inverse of each other. 
\end{proof}

\section{Shifted sums decomposition}
\label{sec:sums}

By combining ideas from \cite{SaSc23}, one can characterize bases as the configurations that admit a decomposition into ``shifted sums''. The decomposition can then be translated to $3$-permutation avoiding $\ddincr$ and $\pat$, yielding a decomposition similar to the one of separable permutations. This is the object of this section.

\subsection{Shifted sums and cuts on bases}

The decomposition of bases relies on the following shifted sum operations, which are illustrated in Figure~\ref{fig:shifted sums}.
\begin{defin}
Let $C_1$ and $C_2$ be two configurations and let $h \in \intval{0, |C_1|}$. The $h$-\emph{vertically-shifted sum} of $C_1$ and $C_2$ is the configuration of size $|C_1| + |C_2|$ defined as $C_1 \vsum_h C_2 \coloneqq ((|C_2|, 0) + C_1) \cup ((0, |C_1|-h)+C_2)$. Similarly, we define the $h$-\emph{horizontally-shifted sum} as $C_1 \hsum_h C_2 \coloneqq ((0,|C_2|)+ C_1) \cup ((h, 0)+C_2)$ and the $h$-\emph{diagonally-shifted sum} as $C_1 \dsum_h C_2 \coloneqq C_1 \cup ((|C_1|-h,h)+ C_2)$. 

If $C$ is a configuration, a pair $(C_1, C_2)$ of subconfigurations of $C$ is a vertical (resp. horizontal, resp. diagonal) \emph{cut} of $C$ if $C$ is a vertically (resp. horizontally, resp. diagonally) shifted-sum of $C_1$ and $C_2$. The \emph{position} of the cut is the coordinate of the line separating $C_1$ and $C_2$, which is equal to the size of $C_2$.
\end{defin}

\begin{figure}[ht!]
    \centering
    \includegraphics[page=4, width=\textwidth]{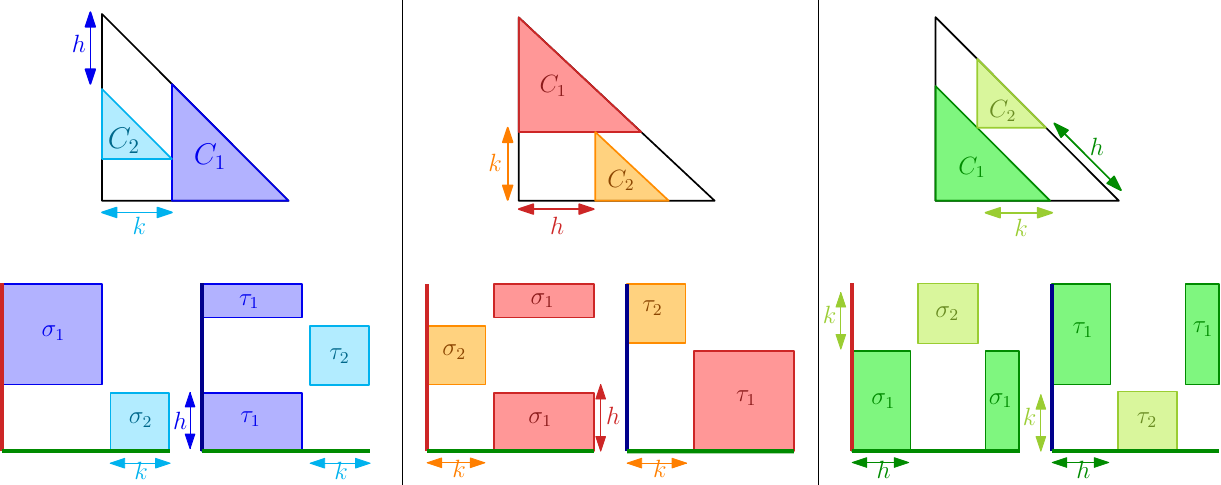}
    \caption{Shifted sums of configurations and $3$-permutations. \\ Left: Vertical sum. Middle: Horizontal sum. Right: Diagonal sum. \\ Parameters $k$ and $h$ are preserved through $\Gamma$.}
    \label{fig:shifted sums}
\end{figure}

The following result derives from arguments in \cite{SaSc23}, although we did not state it in this form there. 
\begin{prop}
\label{prop:basis cut}
    A configuration is a triangle basis if and only if it is either a single point or a shifted sum of two triangle bases.
\end{prop}
\begin{proof}
    Let $B_1, B_2$ be two bases. Let $C$ be a shifted sum of $B_1$ and $B_2$. Then each $B_i$ fills a triangle of its size and those triangles touch, so $C$ fills the smallest triangle containing them, which has size $|B_1|+|B_2|=|C|$ (see \cite[proof of Lemma 2]{SaSc23}). Therefore $C$ is also a basis.

    For the converse, let $B$ be a basis. According to \cite[Proof of Lemma 2]{SaSc23}, one can fill $B$ by ``sequentially merging triangles'' as follows:
    \vspace{-0.1cm}
    \begin{itemize}
    \itemsep=-2pt
        \item[-] Start by considering each point as a triangle of size $1$.
        \item[-] Choose two triangles that touch and fill completely the smallest triangle that contains them.
    \end{itemize}
    It is easy to see that when two triangles touch, it is indeed possible to fill the smallest triangle that contains them. Moreover, if $T$ is a triangle filled by this process, then the subconfiguration $C$ of $B$ contained in $T$ is also a basis. Indeed, as $B$ is a basis, $C$ is sparse, so its the size is at most the size of $T$. Furthermore, by construction $C$ fills, so $|C| = |T|$ and $C$ is a basis. Finally, consider the last two triangles $T_1$ and $T_2$ that are merged by the process and denote by $C_i$ the subconfiguration of $B$ contained in $T_i$. Observe that since both $B$ and the $C_i$ are bases, $T_1$ and $T_2$ cannot intersect (otherwise $B$ would not fill) and their size must sum to the size of $B$, so all the points of $B$ are either in $C_1$ or in $C_2$. This implies that $B$ is a shifted sum of $C_1$ and $C_2$.
\end{proof}

\begin{rem}
    Observe that a given basis can admit several such cuts. For example, the line $L_n=\intval{0,n-1} \times \{0\}$ can be obtained as the diagonal sum of shift $0$ or as the vertical sum of shift $k$ of a line of size $k$ and a line of size $n-k$ for any $1 \leqslant k < n$. 
\end{rem}

\subsection{Shifted sums and cuts on 3-permutations}

We now define the shifted sums on $3$-permutations and show that our bijection transports the cuts. 
\begin{defin}
Let $\colpair{\sigma_1}{\tau_1} \in \S_{k_1}^2$ and $\colpair{\sigma_2}{\tau_2} \in \S_{k_2}^2$ be two $3$-permutations and let $h \in \intval{0,k_1}$. Their $h$-shifted sums are defined as the $3$-permutations $\sigtau$ of size $n = k_1+k_2$ obtained by inserting the diagram $\colpair{\sigma_2}{\tau_2}$ into the one of $\colpair{\sigma_1}{\tau_1}$ at the following position:
\vspace{-0.1cm}
\begin{itemize}
\itemsep=0pt
    \item Vertical sum: Put $\colpair{\sigma_2}{\tau_2}$ in the block $\intval{k_1+1, n}\times \intval{1,k_2} \times \intval{h+1, h+k_2}$.
    \item Horizontal sum: Put $\colpair{\sigma_2}{\tau_2}$ in the block $\intval{1,k_2}\times\intval{h+1, h+k_2}\times\intval{k_1+1,n}$. 
    \item Diagonal sum: Put $\colpair{\sigma_2}{\tau_2}$ in the block $\intval{h+1, h+k_2} \times \intval{k_1+1, n} \times \intval{1, k_2}$.
\end{itemize}
Theses sums are illustrated in Figures~\ref{fig:shifted sums} (with projections) and \ref{fig:3D sums} (in 3D).
As for configurations, we say that a $3$-permutation admits a cut if it is a shifted sum of two $3$-permutations. 
\end{defin}
\begin{figure}[ht!]
    \centering
    \includegraphics[page=8, width=0.8\textwidth]{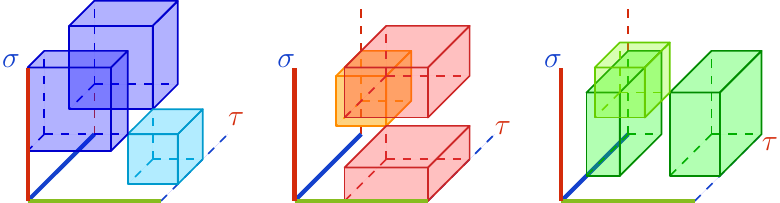}
    \caption{Shifted sums of $3$-permutations. \\ Left: Vertical sum. Middle: Horizontal sum. Right: Diagonal sum.}
    \label{fig:3D sums}
\end{figure}

Cuts are transported by $\Gamma$ (and so $\Psi$) as follows:

\begin{lemma}
\label{lem:cut transfer}
    Let $\sigtau \in Av_n(\ddincr)$. Fix a direction and let $\usum \in \{\vsum, \hsum, \dsum\}$ be the corresponding shifted sum operator. Then $\sigtau$ admits a cut $\sigtau=\colpair{\sigma_1}{\tau_1} \usum_h \colpair{\sigma_2}{\tau_2}$ in that direction, at a given position $k$ with given shift $h$ if and only if $\Gamma(\sigtau)$ does. In that case, ${\Gamma(\colpair{\sigma_1}{\tau_1} \usum_h \colpair{\sigma_2}{\tau_2})} = \Gamma(\colpair{\sigma_1}{\tau_1}) \usum_h \Gamma(\colpair{\sigma_2}{\tau_2})$.
\end{lemma} 
\begin{proof}
    Let $\sigtau \in Av_n(\ddincr)$. First assume that $\Gamma(\sigtau) = C_1 \usum_hC_2$ for some configurations $C_1$ and $C_2$, some sum $\usum$ and some shift $h\in\intval{0,|C_1|}$. Observe that one can build a lozenge tiling of $\Gamma(\sigtau)$ by tiling the triangles containing $C_1$ and $C_2$ respectively, and then filing the two remaining large rhombi with rhombi in the same direction (see Figure~\ref{fig:cut psi}). In fact, since $\Gamma(\sigtau)$ is a basis, this is the only possible lozenge tiling, and in particular $C_1$ and $C_2$ admit a unique tiling and are therefore bases. Now, assume for instance that the sum is vertical. Then observe that by construction of $\Psi$, we have $\sigtau = \Psi(C_1)\vsum_h\Psi(C_2)$. Indeed, denoting $\colpair{\sigma_1}{\tau_1} = \Psi(C_1)$, $\colpair{\sigma_2}{\tau_2} = \Psi(C_2)$ and $n = |C_1|+|C_2|$, for $i \in \intval{|C_1|+1, n}$ we have $\sigma(i) = \sigma_2(i)$ and $\tau(i) = \tau_2(i) +h$; and for $i \in \intval{1, |C_1|}$, we have $\sigma(i) = \sigma_1(i) + |C_2|$ and then $\tau(i) = \tau_1(i)$ if $\tau_1(i) \leqslant h$ and $\tau(i) = \tau_1(i) + |C_2|$ if $\tau_1(i) > h$ (see Figure~\ref{fig:cut psi}, left). The symmetric statement holds for the two other directions.

    \begin{figure}[ht]
        \centering
        \includegraphics[page=6,width=0.95\linewidth]{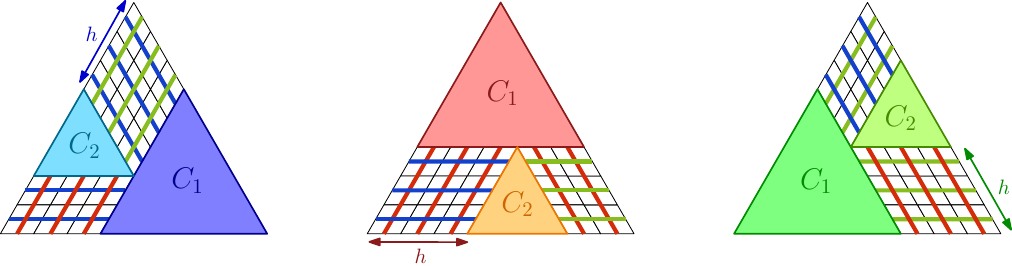}
        \caption{The unique lozenge tiling of the shifted sum of two bases, and how it wires the associated $3$-permutation. \\ Left: Vertical sum. Middle: Horizontal sum. Right: Diagonal sum.}
        \label{fig:cut psi}
    \end{figure}
    
    Now, assume $\sigtau$ can be written as a shifted sum of two $3$-permutations $\colpair{\sigma_1}{\tau_1}$ and $\colpair{\sigma_2}{\tau_2}$.
    
    First assume this sum is vertical and denote by $\ell$ the size of $\colpair{\sigma_1}{\tau_1}$. Then for all $i \in \intval{1,\ell}$ we have $R_\sigma(i) = R_{\sigma_1}(i) \cup \intval{\ell+1, n}$ and $L_\tau(i) = L_{\tau_1}(i)$ so the first $\ell$ points of $\Gamma(\sigtau)$ are $\Gamma(\colpair{\sigma_1}{\tau_1})+(n-\ell, 0)$ ; and for all $i \in \intval{\ell+1, n}$, $R_\sigma(i) = R_{\sigma_2}(i)$ and $L_\tau(i) = L_{\tau_1}(i) + \ell-h$ so the last $n-\ell$ points of $\Gamma(\sigtau)$ are $\Gamma(\colpair{\sigma_2}{\tau_2})+(0, \ell-h)$. This is exactly $\Gamma(\colpair{\sigma_1}{\tau_1} \vsum_h \colpair{\sigma_2}{\tau_2}) = \Gamma(\colpair{\sigma_1}{\tau_1}) \vsum_h \Gamma(\colpair{\sigma_2}{\tau_2})$.

    The horizontal case is obtained symmetrically so let us now assume the cut is diagonal. Such a cut separates $\intval{1,n}$ into $3$ parts, $\intval{1, h}$, $\intval{h+1, h+k}$ and $\intval{h+k+1, n}$. If $i \notin \intval{h+1, h+k}$, then $R_\sigma(i) = R_{\sigma_1}(i)$ and $L_\tau(i) = L_{\tau_1}(i)$ so the image of these points is exactly $\Gamma(\colpair{\sigma_1}{\tau_1})$. If $i \in \intval{h+1, h+k}$ then $R_\sigma(i) = R_{\sigma_2}(i) \cup \intval{h+k+1, n}$ and $L_\tau(i) = L_{\tau_2}(i) \cup \intval{1, h}$ so the image of these points through $\Gamma$ is $\Gamma(\colpair{\sigma_2}{\tau_2}) + (n-k-h, h)$. That is $\Gamma(\colpair{\sigma_1}{\tau_1} \dsum_h \colpair{\sigma_2}{\tau_2}) = \Gamma(\colpair{\sigma_1}{\tau_1}) \dsum_h \Gamma(\colpair{\sigma_2}{\tau_2})$.
\end{proof}    

Combining this lemma with Proposition~\ref{prop:basis cut}, we obtain the following characterization of $3$-permutation avoiding $\ddincr$ and $\pat$.

\begin{thm}
\label{thm:permut carac}
    A $3$-permutation avoiding $\ddincr$ and $\pat$ is either $(\blue{1},\red{1})$ or a shifted sum of two permutations in $Av(\ddincr, \pat)$.
\end{thm}

\begin{rem}
    This characterization of $\A$ can be seen as an analogue to the result of Asinowski and Mansour for separable $d$-permutations \cite{AsMa10}, with a more general notion of cuts.
\end{rem}

\begin{rem}
    It is possible to prove directly that every $3$-permutation avoiding $\ddincr$ and $\pat$ admit a cut by induction, but this requires a tedious case by case study.
\end{rem}

\section{Solitaire}
\label{sec:solitaire}

In this section, we recall a dynamical system on configurations called the \emph{triangle solitaire}, which was introduced by Salo in~\cite{Sa22}, and we extend it to $3$-permutations. This dynamical system presents nice properties that make it useful for random generation purposes, or to extend our bijection to other pattern avoiding permutations classes.

\subsection{Solitaire on the grid}

\begin{defin}
    Let $P$ and $Q$ be sets of points with the same cardinality. There is a \emph{solitaire move} from $P$ to $Q$, denoted $P \move Q$, if there is a position $(x,y) \in \Z^2$ such that their symmetric difference $P \Delta Q$ is a subset of size $2$ of $\{(x,y), (x+1, y), (x, y+1)\}$, i.e. $P\setminus Q = \{p\}$ and $Q\setminus P = \{q\}$ with $p$ and $q$ two points that are in a same size $2$ triangle (see Figure~\ref{fig:solit grid}).
    If $P$ is a configuration, its \emph{orbit}, denoted $\orb(P)$, is the set of configurations reachable from it using solitaire moves.
\end{defin}

Intuitively, seeing $\Z^2$ as a playing grid and $P$ and $Q$ as marbles on this grid, when two marbles are adjacent a solitaire move allows us to move one using the other as a ``pivot''. 

\begin{rem}
    On lozenge tilings, solitaire move then consist in rotating a size $2$ triangle containing two unit triangle and a rhombus. This is a restriction of the \emph{trapezoid flips} studied briefly in \cite[Section 7]{MixedSubdivYY2026}. 
\end{rem}

\begin{figure}[ht]
    \centering
    \includegraphics[page=1, width=0.6\textwidth]{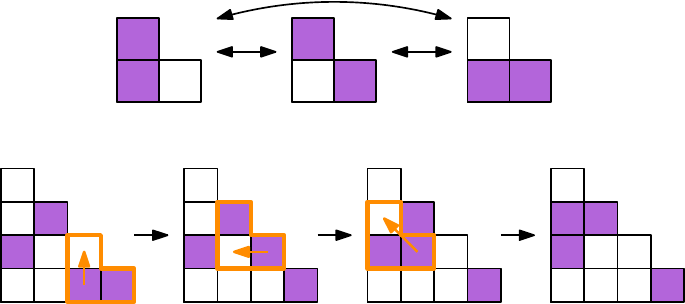}
    \caption{Top: The solitaire moves on the grid.\\ Bottom: An example of moves applied to a pattern.}
    \label{fig:solit grid}
\end{figure}

\begin{thm}[{\cite[Theorem 2]{SaSc23}}]
    \label{thm:orbit}
    For $n \geqslant 1$, the orbit of the line $L_n=\intval{0,n-1} \times \{0\}$ is exactly the set of triangle bases of size $n$.
\end{thm}

\begin{rem}
    The theorem is in fact stronger and states that if two set of points have the same filling and if each triangle in their filling contains the same number of original points in both sets then they are in the same orbit.
\end{rem}

\subsection{Solitaire on 3-permutations}

The solitaire extends naturally to labeled configurations. It can then be pulled back through $\Gamma$ to obtain a dynamical system on $Av(\ddincr)$. 



\begin{defin}
    Let $\sigtau \in \A$ and $1 \leqslant i < j \leqslant n$. A \emph{solitaire move} at $(i,j)$ can be performed in $\sigtau$ if $R_\sigma(i)$ is either $R_\sigma(j)$ or $R_\sigma(j) \cup \{j\}$ and $L_\tau(j)$ is either $L_\tau(i)$ or $L_\tau(i) \cup \{i\}$ (see Figure~\ref{fig:solit moves}). A solitaire move then consists in exchanging the coordinates of $i$ and $j$ along one of the three axes in the diagram of $\sigtau$, i.e. exchanging the values of $\sigma(i)$ and $\sigma(j)$ or $\tau(i)$ and $\tau(j)$ or both (which is the same as exchanging $i$ and $j$). See Figure~\ref{fig:solit permut} for an example.
\end{defin}

\begin{figure}[ht]
    \centering
    \includegraphics[page=2, width=0.9\textwidth]{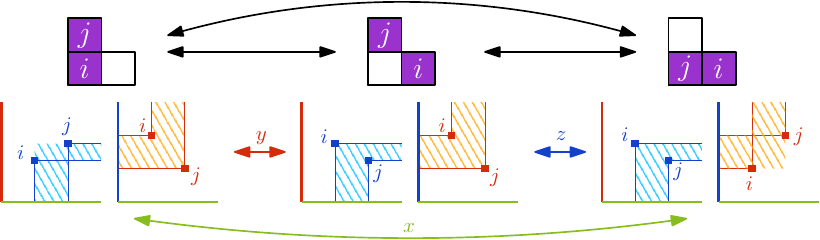}
    \caption{The solitaire moves on the grid and the corresponding situation on $3$-permutations avoiding $\ddincr$, with $i<j$. Hashed areas must be empty for the move to be allowed. One moves from one position to the other by exchanging the coordinates of points $i$ and $j$ along the axis indicated in the diagram.}
    \label{fig:solit moves}
\end{figure}

\begin{figure}[ht]
    \centering
    \includegraphics[page=3, width=\textwidth]{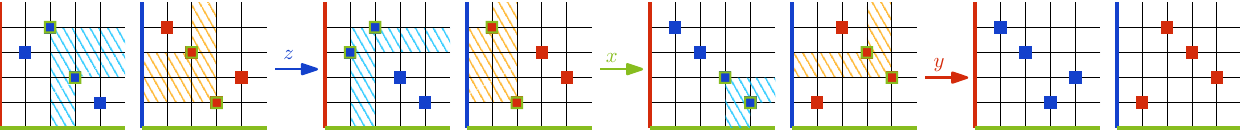}
    \caption{Examples of moves on a $3$-permutation.}
    \label{fig:solit permut}
\end{figure}

\begin{prop}
\label{prop:solit corresp}
    Let $\sigtau \in Av_n(\ddincr)$. Then for all $1\leqslant i < j \leqslant n$ a solitaire move can be performed at $(i,j)$ if and only if one can be performed on $(p_i, p_j)$ in the configuration $\Gamma(\sigtau)$. Furthermore if a solitaire move sends $\sigtau$ to $\colpair{\sigma'}{\tau'}$ then it also sends $\Gamma(\sigtau)$ to $\Gamma(\colpair{\sigma'}{\tau'}$ and reciprocally.
\end{prop}
\begin{proof}
    Let $1 \leqslant i<j \leqslant n$. If $i$ and $j$ satisfy the condition for a solitaire move in $\sigtau$, then $\x(i) \in \{\x(j), \x(j)+1\}$ and $\y(j) \in \{\y(i), \y(i)+1\}$. So $p_i$ and $p_j$ are in a triangle of size $2$, and since $\Gamma(\sigtau)$ is sparse according to Theorem~\ref{thm:sparse}, the third point of this triangle is not in $\Gamma(\sigtau)$. So a solitaire move can be performed on $p_i$ and $p_j$ in $\Gamma(\sigtau)$. Then, it is clear from the equalities on the inversion sets that performing a solitaire move in $\sigtau$ changes the coordinates of $p_i$ and $p_j$ in $\Gamma(\sigtau)$ as announced.

    \smallskip
    For the converse, assume a solitaire move can be performed on $(p_i, p_j)$. Observe that since $\sigtau$ avoids $\ddincr$ either $x_i > x_j$ or $y_i < y_j$ so we know which point is $p_i$ and which is $p_j$: they must be arranged as in Figure~\ref{fig:solit moves}. 
    
    First, let us consider the case where $x_j=x_i$ and $y_j=y_i+1$ (left case). To have $x_j=x_i$, one must have $\sigma(i) < \sigma(j)$ and $|\Rs(i)\setminus\Rs(j)|=|\Rs(j)\setminus\Rs(i)|$. As $\sigtau$ avoids $\ddincr$, this implies $\tau(i) > \tau(j)$ and then since $y_j=y_i+1$ we must have $\Lt(j) = \Lt(i)\cup\{i\}$. This means that the dashed area in the diagram of $\tau$ is empty. Now lets go back to $\sigma$. If there were an index $i<k<j$ such that $\sigma(k) < \sigma(j)$ then we should have $\tau(k) > \tau(j)$, which we just proved to be impossible. So $\Rs(i)\setminus\Rs(j) = \varnothing$ and then as $|\Rs(i)\setminus\Rs(j)|=|\Rs(j)\setminus\Rs(i)|$ we have $\Rs(i)=\Rs(j)$. The case $x_i=x_j+1$ and $y_i=y_j$ is handled symmetrically. 
    
    Now assume $x_i=x_j+1$ and $y_j=y_i+1$. As $\sigtau$ avoids $\ddincr$, we must have either $\sigma(i) > \sigma(j)$ and so $\Rs(i) = \Rs(j)\cup\{j\}$ or $\tau(i) > \tau(j)$ and so $\Lt(j) = \Lt(i)\cup\{i\}$. Then by considering a point $k$ between $i$ and $j$, we obtain that both must hold with the same argument as in the previous cases. 

    Lastly, the simplest way to check that a solitaire move on the configuration has the announced effect on the permutation is to look at a lozenge tiling of the configuration. Indeed, performing a solitaire move changes which of the lines of $\Yline(i)$ and $\Yline(j)$ intersect (see Figure~\ref{fig:solit lozenge}). Then recall from the proof of Theorem~\ref{thm:image psi} that this corresponds to composing $\sigma$, $\tau$ or both with the transposition $(i,j)$, which is exactly the corresponding solitaire move on $\sigtau$. 
    \begin{figure}[ht]
    \centering
    \includegraphics[page=4, width=0.8\textwidth]{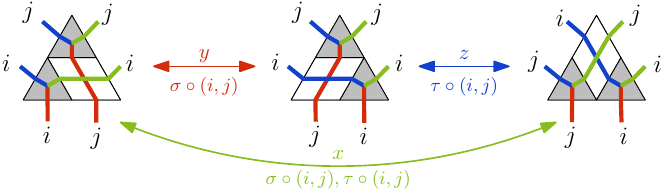}
    \caption{The solitaire moves on lozenge tilings and their effect on the associated $3$-permutation.}
    \label{fig:solit lozenge}
\end{figure}
\end{proof}

Denote by $\desc_n$ the permutation $n(n-1)(n-2)\ldots21$. The mapping $\Gamma$ translates Theorem~\ref{thm:orbit} to $3$-permutations as follows. 

\begin{thm}
    For all $n\geqslant 1$, $\orb((\desc_n, \id_n)) = Av_n(\ddincr, \pat)$.
\end{thm}

\subsection{An orbit correspondence}

A configuration $C$ is \emph{well labeled} if its points are labeled from $1$ to $n$ in such a way that for all $i \in \intval{1,n}$ we have $x_i \leqslant n-i$ and $y_i < i$ and for all $1 \leqslant i < j \leqslant n$, either $x_i > x_j$ or $y_i < y_j$. The first condition corresponds to $C$ being the image of a $3$-permutation through $\Gamma$ and the second was established to be necessary for $C$ to be the image of a $3$-permutation avoiding $\ddincr$ (see Proposition~\ref{prop:image size}).

\begin{lemma}
    Let $C \neq C'$ be two well labeled sparse configurations with the same support. Then $C$ and $C'$ are not in the same solitaire orbit.
\end{lemma}

\begin{proof}
    Assume the only difference between $C$ and $C'$ is an exchange of labels $i$ and $j$ with $i<j$. Then since they are well labeled, without loss of generality $x_i \leqslant x_j$ and $y_i \leqslant y_j$ in $C$. Consider a sequence of solitaire moves that would exchange $p_i$ and $p_j$. It requires $a = x_j - x_i$ points with label greater than $j$ to move $p_j$ left (no less since using the same point twice requires to moved it to the left too) and $b = y_j - y_i$ points with label smaller than $i$ to move $p_i$ up. That makes $a+b+2$ points in the smallest triangle that contains $p_i$ and $p_j$, but this triangle only has size $a+b+1$, so that is impossible since all configurations in the orbit of a sparse configuration are also sparse (see the characterization of the solitaire orbits in \cite[Theorem 3]{SaSc23}, which generalizes Theorem~\ref{thm:orbit}).

    This argument can then be generalized to a larger permutation of labels.
\end{proof}

\begin{corol}
    $\Gamma$ induces a bijection on each solitaire orbit of $Av_n(\ddincr)$ with a solitaire orbit of sparse configurations.
\end{corol}
\begin{proof}
    The previous lemma gives injectivity, and surjectivity is clear as one can use solitaire to build a pre-image for each configuration in the target orbit.
\end{proof}

\section{Discussions}
\label{sec:discu}

\subsection{Enumeration}
\label{sub:bounds}

We now know that $\A$ and the set of triangle bases of size $n$ have the same cardinality, but the question of its value remains open. 

At the time this paper was written, the best known bounds were the following bounds from \cite{SaSc23}:
\[3n!\leqslant |\B_n| \leqslant c\left(\frac{e}2\right)^nn^{n -\frac52} \text{ with } c >0.\]

They transfer to the class $Av_n(\ddincr, \pat)$ through our bijection. To our knowledge, those were the first bounds on this class.
The lower bound can be obtained directly on $3$-permutations by noticing that for all $\tau \in \S_n$, $\colpair{\desc_n}{\tau} \in \A$, and so are its two rotations.

\hfill

The decomposition into shifted sums from Section~\ref{sec:sums} can be used to write equations on the generating function of triangle bases or $3$-permutations. These equations could then be used to obtain more precise enumerative results. In a following paper with Elvey Price and Thévenin \cite{AsymptBases26}, we use the shifted-sum decomposition to derive the asymptotic behavior of the numbers $|\B_n|$, and show that they exhibit a stretched exponential behavior. More precisely,
\[ |\B_n| \sim c n!e^{\sqrt{12n}}n^{5/12}\]
for some $c\approx0.00098107546$.

\subsection{Random sampling}
\label{sub:sampling}

The solitaire defines a Markov chain on triangle bases. The diameter of the reconfiguration graph is $O(n^3)$ \cite{SaSc23}, and we conjecture that its mixing time is also \emph{polynomial} since the graph has strong connectivity properties. This would allow us to sample uniformly triangle bases, and hence $3$-permutations of $\A$ by our bijection.

\subsection{Generalization of the constructions}
\label{sub:other classes}

Our mapping $\Gamma$ is well defined and invertible on all of $Av(\ddincr)$ if we consider labeled configurations, so it might be used to study bijectively other subclasses of $Av(\ddincr)$. An interesting question would be whether the other orbits for the solitaire on permutations can be described as pattern related classes, for instance pattern avoidance classes, pattern covering classes or a mix of the two.

The mapping $\Gamma$ could also be generalized to $d$-permutations by using each permutation to compute a coordinate, the image set would then be configurations of dimension $d-1$. Candidates for interesting image sets would be bases of higher dimensional TEP-subshift \cite{SolitIndep[SaSc25]} or higher dimensional fine mixed subdivisions. The mapping $\Psi$ has already been generalized to the later (see for instance \cite{AcyclicSystAC2013}), but the image permutations set is harder to describe. In particular, it still not known whether this set can be described as a class of permutations avoiding a set of patterns of length $2$ or not.

\subsection{GD-flips on permutations}
\label{sub:GD-flips}

It would also be interesting to investigate how performing a GD-flip on a lozenge tiling $\loz$ changes the $3$-permutation $\Psibis(\loz)$. This might also help us to understand when do two $3$-permutations have the same image through $\Gamma$. 

In Figure~\ref{fig:psi_not_injective}, one can see that performing a GD-flip of size $0$ does not change the $3$-permutation. This motivates the following conjecture:
\begin{conj}
    Two lozenge tilings have the same image through $\Psibis$ if and only if they differ by a sequence of size $0$ GD-flips.
\end{conj}
The consequence on the $3$-permutation of a larger flip seems more difficult to grasp. With the same notations as in Figure~\ref{fig:GD-flip}, it seems that the points $i$, $j$ and $k$ see respectively their $z$, $x$ and $y$ coordinate decreased by at least the size of the GD-cycle, but quantifying the ``at least'' seems to require a good knowledge of the permutation outside of the flip.

A related question is ``when do an occurrence of $\pat$ in the $3$-permutation corresponds to a GD-cycle in the lozenge tiling ?''. It is clear that not all occurrences of $\pat$ correspond to GD-cycle (see Figure~\ref{fig:flip issues}). Looking at some examples, one can conjecture that there is a notion of ``minimal'' occurrence of $\pat$ in the permutation that would correspond to occurrences of $\pat$ corresponding to GD-cycles.
\begin{figure}[ht]
    \centering
    \includegraphics[page=7, height=3.5cm]{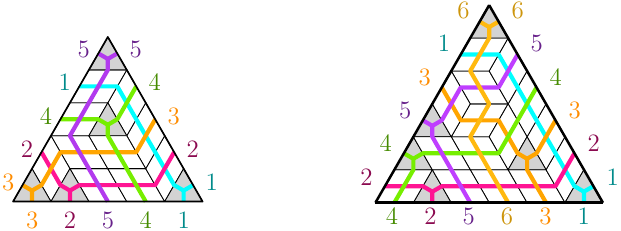}
    \caption{\textbf{Left:} $(1,2,5)$ and $(1,3,5)$ are two occurrences of $\pat$ but only $(1,3,5)$ can be flipped.
    \textbf{Right:} There are five occurrences of $\pat$: $(1,2,6)$, $(1,4,6)$, $(1,5,6)$, $(3,4,6)$ and $(3,5,6)$. The last three correspond to a size $0$ GD-flip, which does not change the $3$-permutation. A sequence of size $0$ GD-flips eventually allows us to perform a size $1$ GD-flip on $(1,2,6)$.}
    \label{fig:flip issues}
\end{figure}

\subsection{Large random permutations} 
\label{sub:permuton}

The combinatorial (and especially bijective) understanding of pattern avoiding permutations gives detailed information about their structure, which can be applied to study the properties of large random objects. For instance, the theory of \emph{permutons} describes the scaling limit of many classes of pattern avoiding permutations (see e.g.~\cite{permuton}).

One can ask the same questions about large random $3$-permutations \cite{BoLi25}, perhaps our bijection or the decomposition into shifted sums can help study the properties of large random permutations of $Av_n(\ddincr,\pat)$.

\section*{Acknowledgments}

The author wishes to thank Nicolas Bonichon for providing the library used to generate the $3$-permutations Figure~\ref{fig:forb pattern} as well as other visualization tools and useful references on $3$-permutations. They also wish to thank Yuan Yao for pointing to the reference to the reverse bijection, as well as interesting discussions on GD-flips and potential generalizations. Finally, they thank Baptiste Louf for its help to make the article clearer as well as the anonymous reviewer who found a simplification in the proof of Theorem~\ref{thm:sparse}.

\subsection*{Founding}

This work was partially supported by the ANR project Combiné (ANR-19-CE48-0011).

\bibliographystyle{alpha}
\bibliography{PermutBasis}

@article {Sa22,
    AUTHOR = {Salo, Ville},
     TITLE = {Cutting corners},
  JOURNAL = {Journal of Computer and System Sciences},
    VOLUME = {128},
      YEAR = {2022},
     PAGES = {35--70},
   MRCLASS = {37B51 (37B15 37B52 52B55 68U05)},
  MRNUMBER = {4407688},
       DOI = {10.1016/j.jcss.2022.03.001}
}

@Article{permuton,
 Author = {Bassino, Fr{\'e}d{\'e}rique and Bouvel, Mathilde and F{\'e}ray, Valentin and Gerin, Lucas and Maazoun, Micka{\"e}l and Pierrot, Adeline},
 Title = {Universal limits of substitution-closed permutation classes},
 Journal = {Journal of the European Mathematical Society (JEMS)},
 Volume = {22},
 Number = {11},
 Year = {2020},
 Pages = {3565--3639},
 Language = {English},
 DOI = {10.4171/JEMS/993},
 zbMATH = {7286839},
 Zbl = {1469.60039}
}

@inproceedings{SaSc23,
  title={Triangle solitaire},
  author={Salo, Ville and Schabanel, Juliette},
  booktitle={International Workshop on Cellular Automata and Discrete Complex Systems},
  pages={123--136},
  year={2023},
  organization={Springer}
}

@article{BoMo22,
    author = {Bonichon,Nicolas and Morel, Pierre-Jean},
    title = {Baxter d-permutations and other pattern avoiding classes},
    journal = {Journal of Integer Sequences},
    year = {2022},
    volume = {25}
}

@article{AsMa10,
	Author = {Asinowski, Andrei and Mansour, Toufik},
	Doi = {10.1007/s00026-010-0043-8},
	Journal = {Annals of Combinatorics},
	Number = {1},
	Pages = {17--43},
	Title = {Separable d-Permutations and Guillotine Partitions},
	Ty = {JOUR},
	Volume = {14},
	Year = {2010},
}

@book{Kit11,
author = {Kitaev, Sergey},
year = {2011},
month = {01},
title = {Patterns in Permutations and Words},
doi = {10.1007/978-3-642-17333-2},
publisher ={Springer Berlin}
}

@article{Ber66,
  title={The undecidability of the domino problem},
  author={Robert L. Berger},
  fjournal={Memoirs of the American Mathematical Society},
  journal = {Mem. Amer. Math. Soc.},
  year={1966},
  volume = {66}
}

@article{Kari05,
title = {Theory of cellular automata: A survey},
journal = {Theoretical Computer Science},
volume = {334},
number = {1},
pages = {3-33},
year = {2005},
doi = {https://doi.org/10.1016/j.tcs.2004.11.021},
author = {Jarkko Kari}
}

@book{Bona,
  title={Combinatorics of Permutations},
  author={B\'ona, Mikl\'os},
  year={2004},
  publisher = {CRC Press}
}

@incollection {Ho16,
    AUTHOR = {Hochman, Michael},
     TITLE = {Multidimensional shifts of finite type and sofic shifts},
 BOOKTITLE = {Combinatorics, words and symbolic dynamics},
    SERIES = {Encyclopedia Math. Appl.},
    VOLUME = {159},
     PAGES = {296--358},
 PUBLISHER = {Cambridge Univ. Press, Cambridge},
      YEAR = {2016},
   MRCLASS = {37B50 (03D80 37B10 68Q45)},
  MRNUMBER = {3525488},
MRREVIEWER = {Benjamin Hellouin de Menibus},
       DOI = {10.1017/CBO9781139924733.010},
}

@article{BoLi25,
      title={High-dimensional permutons: theory and applications}, 
      author={Jacopo Borga and Andrew Lin},
      year={2025},
      journal={ArXiv preprint, arXiv:2412.19730}, 
}

@article{PermutMach,
title = {Permuting machines and priority queues},
journal = {Theoretical Computer Science},
volume = {349},
number = {3},
pages = {309-317},
year = {2005},
issn = {0304-3975},
doi = {https://doi.org/10.1016/j.tcs.2005.07.039},
url = {https://www.sciencedirect.com/science/article/pii/S0304397505005736},
author = {R.E.L. Aldred and M.D. Atkinson and H.P. {van Ditmarsch} and C.C. Handley and D.A. Holton and D.J. McCaughan}
}

@article{dBaxterBMT25,
  title = {Higher Dimensional Floorplans and {{Baxter}} D-Permutations},
  author = {Bonichon, Nicolas and Muller, Thomas and Tanasa, Adrian},
  year = 2025,
  month = apr,
  journal = {ArXiv preprint, arXiv:2504.01116},
  doi = {10.48550/arXiv.2504.01116},
}

@article{AcyclicSystAC2013,
  title = {Acyclic Systems of Permutations and Fine Mixed Subdivisions of Simplices},
  author = {Ardila, Federico and Ceballos, Cesar},
  year = 2013,
  month = apr,
  journal = {Discrete \& Computational Geometry},
  volume = {49},
  number = {3},
  eprint = {1111.2966},
  primaryclass = {math.CO},
  pages = {485--510},
  doi = {10.1007/s00454-013-9485-1},
  archiveprefix = {arXiv}
}

@article{MixedSubdivYY2026,
  title = {Fine {{Mixed Subdivisions}} of a {{Dilated Triangle}}},
  author = {Yao, Yuan and Yudin, Fedir},
  year = 2026,
  month = feb,
  journal = {ArXiv preprint, arXiv:2402.13342},
  doi = {10.48550/arXiv.2402.13342},
}

@article{FlagAB07,
  title = {Flag Arrangements and Triangulations of Products of Simplices},
  author = {Ardila, Federico and Billey, Sara},
  year = 2007,
  month = oct,
  journal = {Advances in Mathematics},
  volume = {214},
  number = {2},
  pages = {495--524},
  doi = {10.1016/j.aim.2007.02.014},
  copyright = {https://www.elsevier.com/tdm/userlicense/1.0/},
  langid = {english}
}

@article{SolitIndep[SaSc25],
  title = {Solitaire of Independence},
  author = {Salo, Ville and Schabanel, Juliette},
  year = 2025,
  month = sep,
  journal = {Natural Computing},
  volume = {24},
  number = {3},
  pages = {431--467},
  doi = {10.1007/s11047-025-10010-3},
  langid = {english}
}

@article{AsymptBases26,
      title={Stretched exponential asymptotics for bases of triangular bootstrap percolation}, 
      author={{Elvey Price}, Andrew and Schabanel, Juliette and Th\'evenin, Paul },
      year={2026},
      journal={ArXiv preprint, arXiv:2607.18901},
}

\end{document}